\documentclass{amsart}
\usepackage{amssymb}
\usepackage{amsfonts}

\newtheorem{theorem}{Theorem}
\theoremstyle{plain}

\newtheorem{conjecture}{Conjecture}
\newtheorem{corollary}{Corollary}

\newtheorem{definition}{Definition}

\newtheorem{lemma}{Lemma}

\newtheorem{problem}{Problem}
\newtheorem{proposition}{Proposition}

\numberwithin{equation}{section}
\input{tcilatex}

\begin{document}
\title[Short Title]{On Commutativity and Finiteness in Groups}
\author{Ricardo N. Oliveira}
\address{Departamento de Matem\'{a}tica, Universidade Federal de Goias,
Goiania, Goias- Brazil}
\email{ricardo@mat.ufg.br }
\author{Said N. Sidki}
\address{Departamento de Matem\'{a}tica, Universidade de Bras\'{\i}lia, Bras%
\'{\i}lia, DF 70910-900, Brazil}
\email{sidki@mat.unb.br}
\date{January 02, 2009}

\begin{abstract}
The second author introduced notions of weak permutablity and commutativity
between groups and proved the finiteness of a group generated by two weakly
permutable finite subgroups. Two groups $H,K$ weakly commute provided there
exists a bijection $f:H\rightarrow K$ which fixes the identity and such that 
$h$ commutes with its image $h^{f}$ for all $h\in H$. The present paper
gives support to conjectures about the nilpotency of groups generated by two
weakly commuting finite abelian groups $H,K$.
\end{abstract}

\thanks{The authors thank Alexander Hulpke for assitance with the double
cosets package in GAP. The second author acknowledges support from the
Brazilian agencies CNPq and FAPDF.}
\maketitle

\section{Introduction}

Commutativity in a group can be depicted by its commutativity graph which
has for vertices the elements of the group and where two elements are joined
by an edge provided they commute. Greater understanding of the commutativity
graph of finite groups has been achieved in recent years and this has had
important applications; see for example \cite{rss}.

For a nontrivial finite $p$-group, it is an elementary and fundamental fact
that it has a non-trivial center and therefore each element in its
commutativity graph is connected to every element in the center. Now,
suppose a finite group $G$ contains a non-trivial $p$-group $A$ such that
every $p$-element in $G$ commutes with \textit{some} non-trivial element in $%
A$. Does it follow that $G$ contains a non-trivial normal $p$-subgroup? In
1976, the second author proved in \cite{sid76} results along such lines for $%
p=2$ and formulated the following conjecture: \textit{if a finite group }$G$%
\textit{\ contains a non-trivial elementary abelian }$2$\textit{-group }$A$%
\textit{\ such that every involution in }$G$\textit{\ commutes with some
involution from }$A$\textit{\ then }$A\cap O_{2}\left( G\right) $\textit{\
is non-trivial}. This was settled in 2006 by Aschbacher, Guralnick and Segev
in \cite{ags07}. The proof made significant use of the classification
theorem of finite simple groups. It was also shown that the result applies
to the Quillen conjecture from 1978 about the Quillen complex at the prime $%
2 $; see \cite{quillen78}.

A configuration which arises in this context is one where the commutation
between elements of $A$ and those of one of its conjugates $B=A^{g}$ is
defined by a bijection. An approach which we had taken in 1980 to such a
weak form of commutativity or permutability was through combinatorial group
theory. The following finiteness criterion was proven in \cite{sid80}.

\begin{theorem}
Let $H,K$ be finite groups having equal orders $n$ and let $f:H\rightarrow K$
be a bijection which fixes the identity. Then for any two maps $%
a:H\rightarrow K,b:H\rightarrow H$, the group 
\begin{equation*}
G\left( H,K;f,a,b\right) =\left\langle H,K\mid hh^{f}=h^{a}h^{b}\text{ for
all }h\in H\right\rangle
\end{equation*}%
is finite of order at most $n\exp (n-1)$, where $h^{f},h^{a},h^{b}$ denote
the images of $h$ under the maps $f,a,b$, respectively.
\end{theorem}

It is to be noted that the proof uses the same argument as that given by I.
N. Sanov for the local finiteness of groups of exponent $4$; see \cite%
{vaug-lee}.

The notion of weak commutativity between $H$ and $K$ by way of a bijection
is formalized by the group 
\begin{equation*}
G\left( H,K;f\right) =\left\langle H,K\mid hh^{f}=h^{f}h\text{ for all }h\in
H\right\rangle \text{.}
\end{equation*}%
When $H$ is isomorphic to $K$, we simplify the notation to $G\left(
H;f\right) $.

If $f$ itself is an isomorphism from $H$ onto $K$ then $G\left( H;f\right) $
is the same group as $\chi \left( H\right) $ and $f$\ as $\psi $ in \cite%
{sid80}. In addition to finiteness, the operator $\chi $ preserves a number
of other group properties such as being a finite $p$-group. Indeed, it was
shown later in \cite{grs86} that more generally, if $H$ is finitely
generated nilpotent then so is $\chi \left( H\right) $.

We are guided in this paper by the following conjecture.

\begin{conjecture}
Let $H$ and $K$ be finite nilpotent groups of equal order and $%
f:H^{\#}\rightarrow K^{\#}$ a bijection. Then $G\left( H,K;f\right) $ is
also nilpotent.
\end{conjecture}

The construction $G\left( H,K;f\right) $ lends itself well to extensions of
groups, as follows. Let $\widetilde{H},\widetilde{K}$ be groups having
normal subgroups $M,N$ respectively and let $\frac{\widetilde{H}}{M}=H$, $%
\frac{\widetilde{K}}{N}=K$. Let $f:H\rightarrow K$, $\alpha :M\rightarrow N$
be bijections both fixing $e$. Then, $f$ and $\alpha $ can be extended in a
natural manner to a bijection $f^{\ast }:\widetilde{H}\rightarrow \widetilde{%
K}$ fixing $e$ such that $\widetilde{G}=G\left( \widetilde{H},\widetilde{K}%
;f^{\ast }\right) $ modulo the normal closure $V$ of $\left\langle
M,N\right\rangle $ is isomorphic to $G\left( H,K;f\right) $. We use this
process to produce an ascending chain of groups of $G\left( H,K;f\right) $
type. For central extensions, we prove

\begin{theorem}
Maintain the above notation. \textit{Suppose }$M,N$\textit{\ are central }%
subgroups of $\widetilde{H},\widetilde{K}$ respectively and that $G\left(
M,N;\alpha \right) $ is abelian. \textit{Then, }$V$ is an abelian group. If
furthermore $H,K$ are abelian then $\left[ V,\widetilde{G}\right] ^{2}$ is
central in $\widetilde{G}$.
\end{theorem}

Although $G\left( H,K;f\right) $ is finite for finite groups, $H,K$, since $%
f $ in general does not behave well with respect to inductive arguments,
methods from finite group theory are difficult to apply. At the present
stage, we have stayed close to the case where the groups $H$ and $K$ are
isomorphic finite abelian groups and more specially to elementary abelian $p$%
-groups $A_{p,k}$ of rank $k$.

A step in the direction of proving nilpotency is

\begin{theorem}
Suppose $A,B$ are finite abelian groups of \ equal order $n$ and let $%
G=G\left( A,B;f\right) $. Then, the metabelian quotient group $\frac{G}{%
G^{\prime \prime }}$ is nilpotent of class at most $n$.
\end{theorem}

The next lemma is a natural first step in classifying the groups $G\left(
H,K;f\right) $ for a fixed pair $\left( H,K\right) $.

\begin{lemma}
Let $a\in Aut\left( H\right) ,b\in Aut\left( K\right) $, and $g=afb$. Then,
the extension $\gamma :G\left( H,K;f\right) \rightarrow G\left( H,K;g\right) 
$ of $h\rightarrow h^{a^{-1}},k\rightarrow k^{b}$ is an isomorphism from $%
G\left( H,K;f\right) \ $onto $G\left( H,K;g\right) $.
\end{lemma}

Let $H$ and $K$ be isomorphic groups by $t:H\rightarrow K$. Then in the
above lemma,%
\begin{equation*}
f=f^{\prime }t,g=g^{\prime }t,b=t^{-1}b^{\prime }t
\end{equation*}%
where $f^{\prime },g^{\prime }\in Sym\left( H^{\#}\right) $ and $b^{\prime
}\in Aut\left( H\right) $; here $H^{\#}$ denotes $H\backslash \left\{
e\right\} $ Therefore, $g^{\prime }=af^{\prime }b^{\prime }$ is an element
of the double coset $Aut\left( H\right) f^{\prime }Aut\left( H\right) $.
Thus, in order to classify $G\left( H;f\right) $ one is obliged to determine
the double coset decomposition $Aut\left( H\right) \backslash Sym\left(
H^{\#}\right) /Aut\left( H\right) $. When the context is clear, we refer to $%
f$ simply by its factor $f^{\prime }$.

Computations by GAP \cite{gap06} produce the following data for abelian
groups of small rank: \newline
for $A_{2,3}$, $SL(3,2)\backslash Sym\left( 7\right) /SL(3,2)$ has $4$
double cosets;\newline
for $A_{2,4}$, $SL(4,2)\backslash Sym\left( 15\right) /SL(4,2)$ has $3374$
double cosets;\newline
for $A_{3,3}$, $PGL\left( 3,3\right) \backslash Sym\left( 13\right)
/PGL\left( 3,3\right) $ has $252$ double cosets.

The groups $G\left( A_{p,k};f\right) $ have the following nilpotency classes
and derived lengths.

\begin{theorem}
(i) Let $A=A_{2,k}$. If $k=3,4$ then $G\left( A;f\right) $ is a $2$-group of
order at most $2^{2^{k}+k-1}$, has class at most $5$ and derived length at
most $3$;\newline
(ii) Let $A=A_{3,3}$. Then $G\left( A;f\right) $ is a $3$-group of order at
most $3^{9}$ and has nilpotency class at most $2$.\newline
\end{theorem}

We discuss a number of other issues. These include the reduction of the
number of relations in $G\left( A;f\right) $, finding a bijection equivalent
to $f$ which is "closer" to being an isomorphism and also concerning $f$
inducing a bijection between the nontrivial cyclic subgroups of $A$.

The paper ends with three general examples. The first is $G\left( A,f\right) 
$ where $A$ is a field, seen as an additive group, and where $f$ corresponds
to the multiplicative inverse. The second example illustrates the
construction of extensions of $G=G\left( A,f\right) $ which are of the same
type as $G$; this produces metabelian $2$-groups having the same order $%
2^{2^{k}+k-1}$ and nilpotency class $k$ as $\chi \left( A_{2,k}\right) $,
but not isomorphic to the latter group. The third is $G=G\left(
A_{2,k},f\right) $ where $f$ corresponds to a transposition of $A_{2,k}^{\#}$%
.

\section{Extensions of groups}

Let $\widetilde{H},\widetilde{K}$ be groups having normal subgroups $M,N$
and let $H$ be a transversal of $M$ in $\widetilde{H}$ with $e\in H$.
Similarly, let $K$ be a transversal of $N$ in $\widetilde{K}$ with $e\in K$.
We identify $H$ with the quotient group $\frac{\widetilde{H}}{M}$ and $K$
with the quotient group $\frac{\widetilde{K}}{N}$. Let $f:H\rightarrow K$, $%
\alpha :M\rightarrow N$ be bijections both fixing $e$. Given a bijection $%
\gamma :M\rightarrow N$ (not necessarily fixing $e$), define $f^{\ast }:%
\widetilde{H}\rightarrow \widetilde{K}$ by 
\begin{equation*}
f^{\ast }:m\rightarrow m^{\alpha },f^{\ast }:mh\rightarrow m^{\gamma }h^{f}%
\text{ if }h\not=e\text{.}
\end{equation*}%
Then $f^{\ast }$ is a bijection which fixes $e$ and $G=G\left( H,K;f\right) $
is an epimorphic image of $\widetilde{G}=G\left( \widetilde{H},\widetilde{K}%
;f^{\ast }\right) $. The natural epimorphisms $\widetilde{H}\rightarrow H,%
\widetilde{K}\rightarrow K$ extend to an epimorphism $\widetilde{G}%
\rightarrow G$ having for kernel $V$, the normal closure of $\left\langle
M,N\right\rangle $ in $\widetilde{G}$.

For any group $L$ let $\nu _{L}$ denote the exponent of $L$ and $\gamma
_{i}\left( L\right) $ the $i$th term of the lower central series of $L$.

Define $\delta :M\rightarrow N,$ $\varepsilon :M\rightarrow M$ by%
\begin{equation*}
m^{\delta }=m^{\alpha }\left( \left( mm^{\alpha \gamma ^{-1}}\right)
^{\gamma }\right) ^{-1},\text{ }m^{\varepsilon }=m\left( \left( m^{\alpha
}m^{\gamma }\right) ^{\gamma ^{-1}}\right) ^{-1}\text{.}
\end{equation*}%
Clearly, $m^{\delta }=e$ if and only if $m=e$ and similarly, $m^{\varepsilon
}=e$ if and only if $m=e$.

\begin{theorem}
Maintain the previous notation. Suppose $M,N$\textit{\ are central subgroups
of }$\widetilde{H},\widetilde{K}$, respectively and that $G\left( M,N;\alpha
\right) $ is abelian. Then, $V$ is an abelian group such that\textit{\ }%
\begin{eqnarray*}
V &=&M+N+\left[ V,\widetilde{G}\right] \text{,} \\
\left[ V,\widetilde{G}\right] &=&\left[ M^{\delta },H\right] =\left[
N^{\varepsilon },H\right] \text{,} \\
\left[ V,\gamma _{i}\left( \widetilde{H}\right) \right] &=&\left[ V,i%
\widetilde{H}\right] ^{2^{i-1}}
\end{eqnarray*}%
for all $i\geq 1$ and 
\begin{equation*}
\nu _{\left[ V,\widetilde{G}\right] }\mid \gcd \left( \nu _{M},\nu _{N},\nu
_{H},\nu _{K}\right) \text{.}
\end{equation*}
Both sets%
\begin{eqnarray*}
&&\left\{ \left( m_{1}^{\delta }\right) ^{-1}\left( m_{2}^{\delta }\right)
^{-1}\left( m_{1}m_{2}\right) ^{\delta }\mid m_{1},m_{2}\in M\right\} , \\
&&\left\{ m^{-1}\left( \left( m^{\delta }\right) ^{\alpha ^{-1}}\right)
^{\varepsilon }\mid m\in M\right\}
\end{eqnarray*}%
are central in $\widetilde{G}$. Furthermore, if $H,K$ are abelian then $%
\left[ V,\widetilde{G}\right] ^{2}$ is central in $\widetilde{G}$.
\end{theorem}

\begin{proof}
(1) Let $m\not=e\not=h$. Then, 
\begin{eqnarray*}
\left[ m,h^{f}\right] &=&\left[ m,m^{\alpha }h^{f}\right] \\
&=&\left[ \left( m^{\alpha \gamma ^{-1}}h\right) m,m^{\alpha }h^{f}\right] \\
&=&\left[ \left( mm^{\alpha \gamma ^{-1}}\right) h,m^{\alpha }h^{f}\right] \\
&=&\left[ \left( mm^{\alpha \gamma ^{-1}}\right) h,\left( \left( mm^{\alpha
\gamma ^{-1}}\right) ^{\gamma }h^{f}\right) ^{-1}m^{\alpha }h^{f}\right] \\
&=&\left[ \left( mm^{\alpha \gamma ^{-1}}\right) h,m^{\delta }\right] =\left[
h,m^{\delta }\right] \text{.}
\end{eqnarray*}

In the same manner,%
\begin{eqnarray*}
\left[ h,m^{\alpha }\right] &=&\left[ mh,m^{\alpha }\right] =\left[
mh,m^{\gamma }h^{f}.m^{\alpha }\right] \\
&=&\left[ mh,m^{\alpha }m^{\gamma }h^{f}\right] \\
&=&\left[ \left( \left( m^{\alpha }m^{\gamma }\right) ^{\gamma
^{-1}}h\right) ^{-1}mh,m^{\alpha }m^{\gamma }h^{f}\right] \\
&=&\left[ m^{\varepsilon },m^{\alpha }m^{\gamma }h^{f}\right] =\left[
m^{\varepsilon },h^{f}\right] \text{.}
\end{eqnarray*}%
Since $\left[ h,m^{\delta }\right] =\left[ m,h^{f}\right] $ and $m^{\delta
}=\left( m^{\prime }\right) ^{\alpha }$ for $m^{\prime }=\left( m^{\delta
}\right) ^{\alpha ^{-1}}$, we obtain $\left[ \left( m^{\prime }\right)
^{\varepsilon },h^{f}\right] =\left[ h,\left( m^{\prime }\right) ^{\alpha }%
\right] $, 
\begin{eqnarray*}
\left[ \left( m^{\prime }\right) ^{\varepsilon },h^{f}\right] &=&\left[
h,\left( m^{\prime }\right) ^{\alpha }\right] , \\
\left[ \left( \left( m^{\delta }\right) ^{\alpha ^{-1}}\right) ^{\varepsilon
},h^{f}\right] &=&\left[ h,m^{\delta }\right] =\left[ m,h^{f}\right] \text{;}
\end{eqnarray*}%
that is, $m^{-1}\left( \left( m^{\delta }\right) ^{\alpha ^{-1}}\right)
^{\varepsilon }$ commutes with $K$ and is therefore central in $\widetilde{G}
$.

(2) We obtain from $\left[ m,h^{f}\right] =\left[ h,m^{\delta }\right] $
that the sets $\left[ m,K\right] ,\left[ H,m^{\delta }\right] $ are equal
and $H,K$ normalize the subgroup $\left\langle \left[ m,K\right]
\right\rangle $. Also, any $m^{\prime }\in M$ commutes with $\left[
h,m^{\delta }\right] \left( =\left[ m,h^{f}\right] \right) \in \left[ M,K%
\right] $. Therefore, $M^{K}\left( =M^{\widetilde{K}}\right) $ is abelian.

Since $\left[ M,K\right] \leq \left[ H,M^{\delta }\right] \leq \left[ N,H%
\right] $ and $\left[ H,N\right] \leq \left[ M^{\varepsilon },K\right] \leq %
\left[ M,K\right] $ we get 
\begin{equation*}
\left[ M,K\right] =\left[ H,M^{\delta }\right] =\left[ H,N\right] =\left[
M^{\varepsilon },K\right] \text{.}
\end{equation*}%
Therefore $V$ is abelian and 
\begin{eqnarray*}
V &=&M+N+\left[ M^{\delta },H\right] \text{,} \\
\left[ V,\widetilde{G}\right] &=&\left[ M^{\delta },H\right] \text{,} \\
\nu _{\left[ V,\widetilde{G}\right] } &\mid &\gcd \left( \nu _{M},\nu
_{N}\right) \text{.}
\end{eqnarray*}

(3) Let $m_{1},m_{2}\in M$. Then, 
\begin{equation*}
\left[ m_{1}m_{2},h^{f}\right] =\left[ h,\left( m_{1}m_{2}\right) ^{\delta }%
\right] \text{,}
\end{equation*}%
\begin{eqnarray*}
\left[ m_{1}m_{2},h^{f}\right] &=&\left[ m_{1},h^{f}\right] ^{m_{2}}\left[
m_{2},h^{f}\right] \\
&=&\left[ m_{1},h^{f}\right] \left[ m_{2},h^{f}\right] \\
&=&\left[ h,\left( m_{1}\right) ^{\delta }\right] \left[ h,\left(
m_{2}\right) ^{\delta }\right] \text{,}
\end{eqnarray*}%
and%
\begin{equation*}
\left[ h,\left( m_{1}m_{2}\right) ^{\delta }\right] =\left[ h,\left(
m_{1}\right) ^{\delta }\right] \left[ h,\left( m_{2}\right) ^{\delta }\right]
\text{.}
\end{equation*}%
Thus, 
\begin{equation*}
\left\{ \left( m_{1}^{\delta }\right) ^{-1}\left( m_{2}^{\delta }\right)
^{-1}\left( m_{1}m_{2}\right) ^{\delta }\mid m_{1},m_{2}\in M\right\}
\end{equation*}%
is central in $\widetilde{G}$ and likewise for 
\begin{equation*}
\left\{ \left( m_{1}^{\varepsilon }\right) ^{-1}\left( m_{2}^{\varepsilon
}\right) ^{-1}\left( m_{1}m_{2}\right) ^{\varepsilon }\mid m_{1},m_{2}\in
M\right\} \text{.}
\end{equation*}

(4) Since any $h\in H$ centralizes $\left\langle M,h^{f}\right\rangle $, we
have for all $m\in M$,%
\begin{equation*}
\left[ m,h^{f},h\right] =e=\left[ m^{\delta },h,h\right]
\end{equation*}%
and so, 
\begin{equation*}
\left[ m^{\delta },h,h\right] =e,\left[ m^{\delta },h^{i}\right] =\left[
m^{\delta },h\right] ^{i}
\end{equation*}%
for all $v\in V$ and all integers $i$. \ Therefore,%
\begin{equation*}
\nu _{\left[ V,\widetilde{G}\right] }\mid \gcd \left( \nu _{N},\nu
_{H}\right) ,\gcd \left( \nu _{M},\nu _{K}\right) \text{.}
\end{equation*}

We calculate for $h_{1},h_{2}\in H$ 
\begin{eqnarray*}
\left[ m^{\delta },\left( h_{1}h_{2}\right) ^{-1}\right] &=&\left[ m^{\delta
},h_{1}h_{2}\right] ^{-1}, \\
\left[ m^{\delta },h_{2}^{-1}h_{1}^{-1}\right] &=&\left[ m^{\delta
},h_{1}h_{2}\right] ^{-1}, \\
\left[ m^{\delta },h_{1}^{-1}\right] \left[ m^{\delta },h_{2}^{-1}\right]
^{h_{1}^{-1}} &=&\left( \left[ m^{\delta },h_{2}\right] \left[ m^{\delta
},h_{1}\right] ^{h_{2}}\right) ^{-1}, \\
\left[ m^{\delta },h_{1}\right] \left[ m^{\delta },h_{2}\right]
^{h_{1}^{-1}} &=&\left[ m^{\delta },h_{2}\right] \left[ m^{\delta },h_{1}%
\right] ^{h_{2}}, \\
\left[ m^{\delta },h_{2}\right] ^{-1}\left[ m^{\delta },h_{2}\right]
^{h_{1}^{-1}} &=&\left[ m^{\delta },h_{1}\right] ^{-1}\left[ m^{\delta
},h_{1}\right] ^{h_{2}}, \\
\left[ m^{\delta },h_{2},h_{1}^{-1}\right] &=&\left[ m^{\delta },h_{1},h_{2}%
\right] \text{.}
\end{eqnarray*}%
Therefore,%
\begin{eqnarray*}
\left[ m^{\delta },h_{1},h_{2}\right] &=&\left[ m^{\delta },h_{2},h_{1}^{-1}%
\right] =\left[ m^{\delta },h_{1}^{-1},h_{2}^{-1}\right] \\
&=&\left[ m^{\delta },h_{1},h_{2}^{-1}\right] ^{-1}, \\
\left[ m^{\delta },h_{1},h_{2}\right] ^{-1} &=&\left[ m^{\delta
},h_{1},h_{2}^{-1}\right] , \\
\left[ m^{\delta },h_{1},h_{2}\right] &=&\left[ m^{\delta },h_{2},h_{1}^{-1}%
\right] =\left[ m^{\delta },h_{2},h_{1}\right] ^{-1}\text{.}
\end{eqnarray*}

Calculate further%
\begin{eqnarray*}
\left[ m^{\delta },h_{1}h_{2}\right] &=&\left[ m^{\delta },h_{2}\right] %
\left[ m^{\delta },h_{1}\right] ^{h_{2}} \\
&=&\left[ m^{\delta },h_{2}\right] \left[ m^{\delta },h_{1}\right] \left[
m^{\delta },h_{1},h_{2}\right] , \\
\left[ m^{\delta },h_{1}h_{2},h_{1}h_{2}\right] &=&\left[ m^{\delta
},h_{2},h_{1}h_{2}\right] \left[ m^{\delta },h_{1},h_{1}h_{2}\right] \\
&&\left[ m^{\delta },h_{1},h_{2},h_{1}h_{2}\right] \\
&=&\left[ m^{\delta },h_{2},h_{1}\right] \left[ m^{\delta },h_{2},h_{1},h_{2}%
\right] \\
&&\left[ m^{\delta },h_{1},h_{2}\right] \\
&&\left[ m^{\delta },h_{1},h_{2},h_{2}\right] \left[ m^{\delta
},h_{1},h_{2},h_{1}\right] .\left[ m^{\delta },h_{1},h_{2},h_{1},h_{2}\right]
\\
&=&\left[ m^{\delta },h_{1},h_{2},h_{1}\right] ^{h_{2}}=e\text{.}
\end{eqnarray*}
We conclude 
\begin{equation*}
\left[ m^{\delta },h_{2},h_{1},h_{1}\right] =e\text{.}
\end{equation*}%
Thus,%
\begin{equation*}
\left[ v,h,h\right] =e
\end{equation*}%
for all $v\in V$.

When $V$ is written additively, the action of $h$ on $V$ can be expressed as%
\begin{equation*}
v\left( -1+h\right) ^{2}=0\text{.}
\end{equation*}%
From this we\ derive the following formulae for the action of $H$ on $V$:

for $h_{1},h_{2},...h_{i}\in H$, 
\begin{eqnarray*}
h_{2}h_{1} &=&-2+2h_{1}+2h_{2}-h_{1}h_{2}, \\
h_{1}^{-1}h_{2}h_{1} &=&-2+2h_{1}+3h_{2}-2h_{1}h_{2}, \\
-1+\left[ h_{2},h_{1}\right] &=&-2\left( -1+h_{1}\right) \left(
-1+h_{2}\right) , \\
-1+\left[ h_{i},...,h_{2},h_{1}\right] &=&\left( -2\right) ^{i-1}\left(
-1+h_{1}\right) \left( -1+h_{2}\right) ...\left( -1+h_{i}\right) \text{.}
\end{eqnarray*}%
In other words,%
\begin{eqnarray*}
\left[ v,\left[ h_{i},...,h_{2},h_{1}\right] \right] &=&\left[
v,h_{1},h_{2},...,h_{i}\right] ^{\left( -2\right) ^{i-1}}, \\
\left[ V,\gamma _{i}\left( \widetilde{H}\right) \right] &=&\left[ V,i%
\widetilde{H}\right] ^{2^{i-1}}\text{.}
\end{eqnarray*}

(5) Suppose $H,K$ abelian. Then $\left[ v,\left[ h_{1},h_{2}\right] \right] =%
\left[ v,h_{1},h_{2}\right] ^{2}=\left[ \left[ v,h_{1}\right] ^{2},h_{2}%
\right] =e$. As $\left[ V,\widetilde{G}\right] =\left[ V,H\right] =\left[ V,K%
\right] $, we conclude that $\left[ V,\widetilde{G}\right] ^{2}$ is central.
\end{proof}

\begin{theorem}
\textit{Suppose in the above, }$\widetilde{H},\widetilde{K}$ are \textit{%
finite groups and let }$M=\left\langle m\right\rangle ,N=\left\langle
n\right\rangle $\textit{\ be cyclic central }subgroups of $\widetilde{H},%
\widetilde{K}$ respectively, each of\ prime order $p$. \textit{Then, }$%
\left\langle M,N\right\rangle ^{\widetilde{G}}$ is an elementary abelian $p$%
-subgroup of rank at most $\left\vert H\right\vert +1$.
\end{theorem}

\begin{proof}
We have $\left[ M,N\right] =\left\{ e\right\} $, $\left[ M,K\right] =\left[
H,N\right] $ elementary $p$-abelian subgroup and $M,N$ centralize $\left[ M,K%
\right] $. Therefore,%
\begin{equation*}
\left\langle M,N\right\rangle ^{\widetilde{G}}=\left\langle M,N\right\rangle
^{G}=\left\langle m^{K},n\right\rangle =\left\langle m,n^{H}\right\rangle
\end{equation*}%
is an elementary abelian $p$-subgroup of rank at most $\left\vert
K\right\vert +1$.
\end{proof}

\section{Metabelian Quotients of $G\left( A,B,f\right) $}

It would be interesting to resolve the question of nilpotency of the
solvable quotients of $G\left( H,K;f\right) $. In the next result we
consider metabelian quotients of $G\left( A,B;f\right) $.

\begin{theorem}
Suppose $A,B$ are finite abelian groups of \ equal order $n$ and let $%
G=G\left( A,B;f\right) $. Then, the metabelian quotient group $\frac{G}{%
G^{\prime \prime }}$ is nilpotent of class at most $n$.
\end{theorem}

\begin{proof}
In a metabelian group $M$, if $u\in M^{\prime }$ and $x_{i}\in M$ $\left(
1\leq i\leq k\right) $ then%
\begin{equation*}
\left[ u,x_{1},x_{2},...,x_{k}\right] =\left[
u,x_{i_{1}},x_{i_{2}},...,x_{i_{k}}\right]
\end{equation*}%
for any permutation of $\left( i_{1},i_{2}...,i_{k}\right) $ of $\left\{
1,2,...,k\right\} $.

We have from the relations of $G,$ 
\begin{eqnarray*}
\left[ a,b\right] &=&\left[ a,\left( a^{f}\right) ^{-1}b\right] \\
&=&\left[ a.\left( b^{f^{-1}}\right) ^{-1},b\right] \text{.}
\end{eqnarray*}%
Let $a\in A,b_{1},b_{2}\in B$. As $a^{f}$ and $b_{1}^{f^{-1}}$commute with
both $a,b_{1}$, it follows that 
\begin{eqnarray*}
\left[ a,b_{1},b_{2}\right] &=&\left[ a,b_{1},\left( a^{f}\right) ^{-1}.b_{2}%
\right] \text{ } \\
\left[ a,b_{1},b_{2}\right] &=&\left[ a.\left( b_{1}^{f^{-1}}\right)
^{-1},b_{1},b_{2}\right] \\
&=&\left[ a,b_{1},\left( a.\left( b_{1}^{f^{-1}}\right) ^{-1}\right)
^{f}b_{2}\right] \text{.}
\end{eqnarray*}%
We observe that if $b_{1}=b_{2}\not=a^{f}$, then $b_{1}\not=\left( a.\left(
b_{1}^{f^{-1}}\right) ^{-1}\right) ^{f}b_{2}$.

Now, we will work in $G$ modulo $G^{\prime \prime }$.

Let $a\in A$, $b_{i}$ $\left( 2\leq i\leq k\right) \in B$. From Witt%
\'{}%
s formula, as $A,B$ are abelian, we have%
\begin{equation*}
\left[ a,b_{2},b_{1}\right] =\left[ a,b_{1},b_{2}\right]
\end{equation*}%
and more generally,

\begin{equation*}
\left[ a,b_{1},b_{2},...,b_{k}\right] =\left[
a,b_{i_{1}},b_{i_{2}},...,b_{i_{k}}\right]
\end{equation*}%
for any permutation of $\left( i_{1},,i_{2}...,i_{k}\right) $ of $\left\{
1,2,...,k\right\} $.

Therefore, if 
\begin{equation*}
\left\{ x_{2},...,x_{k}\right\} =\left\{
b_{2},...,b_{s},a_{s+1},...,a_{k}\right\}
\end{equation*}%
with $b_{2},...,b_{s}\in B$ and $a_{s+1},...,a_{s+k}\in A$ then

\begin{equation*}
\left[ a,b_{1},x_{2},...,x_{k}\right] =\left[
a,b_{1},b_{2},...,b_{s},a_{s+1},...,a_{k}\right] \text{. }
\end{equation*}

Let again $a\in A$, $b_{i}$ $\left( 2\leq i\leq k\right) \in B$. Suppose
that $b_{1},b_{2},..,b_{k-1},a^{f}$ are distinct.

Then, 
\begin{equation*}
b_{k},\left( a.\left( b_{1}^{f^{-1}}\right) ^{-1}\right)
^{f}b_{k},,...,\left( a.\left( b_{k-1}^{f^{-1}}\right) ^{-1}\right) ^{f}b_{k}
\end{equation*}

are $k$ distinct elements of $B$.

Suppose further that $b_{k}=b_{j}$ for some $1\leq j\leq k-1$. Then, for
some $i$,%
\begin{equation*}
b_{ij}^{\prime }=\left( a.\left( b_{i}^{f^{-1}}\right) ^{-1}\right)
^{f}b_{j}\not\in \left\{ b_{1},b_{2},..,b_{k-1}\right\} \text{.}
\end{equation*}

In this manner, we have

\begin{eqnarray*}
\left[ a,b_{1},b_{2},..,b_{k-1},b_{k}\right] &=&\left[
a,b_{1},b_{2},..,b_{k-1},b_{j}\right] \\
&=&\left[ a,b_{i},b_{j},b_{l},..,b_{m}\right] \\
&=&\left[ a,b_{i},b_{ij}^{\prime },b_{l},..,b_{m}\right] \\
&=&\left[ a,b_{1},b_{2},..,b_{k-1},b_{ij}^{\prime }\right]
\end{eqnarray*}%
and $b_{1},b_{2},..,b_{ij}^{\prime }$ are distinct. If $k=n$ then $\left[
a,b_{1},b_{2},..,b_{n}\right] =e$ and $G$ is nilpotent of class at most $n$.
\end{proof}

The limit $n$ obtained in the proof is clearly too large, especially when
compared with available results. Determining the nilpotency degree seems to
stem from a more general problem which can be formulated for commutative
rings.

\begin{problem}
Let $A$ be a free abelian group of rank $m$ generated by $a_{1},...,a_{m}$
and let $A$ act on a torsion-free $\mathbb{Z}$-module $V$. Let $n$ be a
natural number. Define 
\begin{equation*}
S\left( m,n\right) =\left\{ 
\begin{array}{c}
a_{i_{1}}^{l_{1}}a_{i_{2}}^{l_{2}}...a_{i_{s}}^{l_{s}}a_{i_{s}+1}\mid 0\leq
s\leq m-1, \\ 
1\leq i_{1}<i_{2}<...<i_{s}<m, \\ 
1\leq l_{i}\leq n-1%
\end{array}%
\right\} \text{.}
\end{equation*}%
This set corresponds to a choice of a generator for each of the different
cyclic subgroups of order $n$ in $\frac{A}{A^{n}}$.For instance, $S\left(
2,3\right) =\left\{ a_{1},a_{2},a_{1}a_{2},a_{1}^{2}a_{2}\right\} $.Let $f$
be a permutation of $S\left( m,n\right) $ and define 
\begin{equation*}
U\left( m,n;f\right) =\left\{ (1-x)\left( 1-x^{f}\right) \mid \text{ }x\in
S\left( m,n\right) \right\} \text{.}
\end{equation*}%
Suppose that $A$ acts on $V$ such that $U\left( m,n;f\right) =\left\{
0\right\} $. Prove that the action of $A$ on $V$ is nilpotent. Moreover,
that it has nilpotency degree at most $3$ for $n=2$ and degree at most $2$
for $n\geq 3$. \textit{The small bounds have been confirmed by a number of
examples using the Groebner basis applied to }$Q[a_{1},a_{2},..,a_{n}]$%
\textit{\ modulo the ideal generated by }$U\left( m,n;f\right) $\textit{.}
\end{problem}

\section{Reduction of the presentation of $\protect\chi \left( A\right) $}

It is interesting to reduce the number of relations in the definition of $%
G\left( A,f\right) $, particularly for the sake of applications. This is
difficult to carry out in general. We treat here the question for the group $%
\chi \left( A\right) $. Given a generating set $S$ of $A$ then define 
\begin{equation*}
\chi \left( A,S;m\right) =\left\langle A,A^{\psi }\mid \left[ a,a^{\psi }%
\right] =e\text{ for all }a\in \cup _{1\leq i\leq m}S^{i}\right\rangle \text{%
.}
\end{equation*}%
The construction $\chi \left( A,S;1\right) $ does not conserve finiteness in
general. For let $A=A_{2,2}$ be generated by $S=\left\{ a_{1},a_{2}\right\} $%
. Then, on defining $x_{1}=a_{1}a_{2}^{\psi },x_{2}=a_{1}^{\psi }a_{2}$, we
find $\chi \left( A,S;1\right) =\left\langle x_{1},x_{2}\right\rangle A$,
where $\left\langle x_{1},x_{2}\right\rangle $ is free abelian of rank $2$. 

We start with

\begin{proposition}
Let $H$ be a group generated by $x_{1},x_{2},y_{1},y_{2}$ such that $\left[
y_{1},x_{1}\right] =e=\left[ y_{2},x_{2}\right] $. Then \newline
(i)%
\begin{equation*}
\left[ y_{1}y_{2},x_{1}x_{2}\right] =\left[ y_{1}^{y_{2}},x_{2}\right] \left[
y_{2},x_{1}^{x_{2}}\right] \text{;}
\end{equation*}%
(ii) if in addition $\left[ x_{1},x_{2}\right] =\left[ y_{1},y_{2}\right] =e$
holds then%
\begin{equation*}
\left[ y_{1}y_{2},x_{1}x_{2}\right] =\left[ y_{1},x_{2}\right] \left[
y_{2},x_{1}\right] \text{ (*);}
\end{equation*}%
(iii) if furthermore $\left[ y_{1}y_{2},x_{1}x_{2}\right] =e$ holds then $H$
is nilpotent of class at most $2$ with derived subgroup $H^{\prime
}=\left\langle \left[ y_{1},x_{2}\right] \right\rangle $.
\end{proposition}

\begin{proof}
The first two items are shown directly. The last item follows from%
\begin{eqnarray*}
\left[ y_{1},x_{2}\right] &=&\left[ x_{1},y_{2}\right] ,\left[ y_{1},x_{2}%
\right] ^{x_{1}}=\left[ y_{1},x_{2}\right] , \\
\left[ y_{1},x_{2}\right] ^{x_{2}} &=&\left[ x_{1},y_{2}\right] ^{x_{2}}=%
\left[ x_{1},y_{2}\right] \text{.}
\end{eqnarray*}
\end{proof}

\begin{corollary}
Let $A$ be an abelian group generated by $S=\left\{ a_{1},a_{2}\right\} $
and let $G=\chi \left( A,S;2\right) $. Then, $G=\chi \left( A\right) $.
\end{corollary}

For abelian groups $A$ of rank $3$, the situation becomes less simple.

\begin{proposition}
Let $A$ be an abelian group generated by $S=\left\langle a_{i}|1\leq i\leq
3\right\rangle $. Then the following equations hold $G=\chi \left(
A,S;2\right) $: 
\begin{equation*}
\xi =\left[ a_{1}^{\psi }a_{2}^{\psi }a_{3}^{\psi },a_{1}a_{2}a_{3}\right] =%
\left[ a_{1}^{\psi },a_{3}\right] ^{a_{2}^{\psi }}\left[ a_{3}^{\psi },a_{1}%
\right] ^{a_{2}},
\end{equation*}%
\begin{equation*}
\left[ a_{1}^{\psi },a_{3}^{2}\right] ^{\left[ a_{2},\psi \right] }=\left[
a_{1}^{\psi },a_{3}^{2}\right] \text{.}
\end{equation*}%
\newline
\end{proposition}

\begin{proof}
On substituting $x_{1}=a_{1}a_{2},x_{2}=a_{3},y_{1}=a_{1}^{\psi }a_{2}^{\psi
},y_{2}=a_{3}^{\psi }$ in (*) of the previous proposition, we obtain%
\begin{eqnarray*}
\xi &=&\left[ a_{1}^{\psi }a_{2}^{\psi }a_{3}^{\psi },a_{1}a_{2}a_{3}\right]
\\
&=&\left[ a_{1}^{\psi },a_{3}\right] ^{a_{2}^{\psi }}\left[ a_{2}^{\psi
},a_{3}\right] \left[ a_{3}^{\psi },a_{2}\right] \left[ a_{3}^{\psi },a_{1}%
\right] ^{a_{2}} \\
&=&\left[ a_{1}^{\psi },a_{3}\right] ^{a_{2}^{\psi }}\left[ a_{2}^{\psi
},a_{3}\right] \left[ a_{3},a_{2}^{\psi }\right] \left[ a_{3}^{\psi },a_{1}%
\right] ^{a_{2}} \\
&=&\left[ a_{1}^{\psi },a_{3}\right] ^{a_{2}^{\psi }}\left[ a_{3}^{\psi
},a_{1}\right] ^{a_{2}}\text{.}
\end{eqnarray*}

Substitute $a_{1}\leftrightarrow a_{3},a_{2}\leftrightarrow a_{2},\psi
\rightarrow \psi $ above to obtain%
\begin{equation*}
\xi =\left[ a_{3}^{\psi },a_{1}\right] ^{a_{2}^{\psi }}\left[
a_{1},a_{3}^{\psi }\right] ^{a_{2}}\text{.}
\end{equation*}

Therefore, since $\left[ a_{1},a_{3}^{\psi }\right] =\left[ a_{1}^{\psi
},a_{3}\right] $, and $\left\langle a_{1},a_{3},a_{1}^{\psi },a_{3}^{\psi
}\right\rangle $ has class at most $2$, we have 
\begin{eqnarray*}
\left[ a_{1}^{\psi },a_{3}\right] ^{a_{2}^{\psi }}\left[ a_{3}^{\psi },a_{1}%
\right] ^{a_{2}} &=&\left[ a_{3}^{\psi },a_{1}\right] ^{a_{2}^{\psi }}\left[
a_{1},a_{3}^{\psi }\right] ^{a_{2}}, \\
\left[ a_{1}^{\psi },a_{3}\right] ^{a_{2}^{\psi }}\left[ a_{3}^{\psi },a_{1}%
\right] ^{a_{2}} &=&\left[ a_{3},a_{1}^{\psi }\right] ^{a_{2}^{\psi }}\left[
a_{1},a_{3}^{\psi }\right] ^{a_{2}}, \\
\left[ a_{1}^{\psi },a_{3}\right] ^{2a_{2}^{\psi }} &=&\left[ a_{1}^{\psi
},a_{3}\right] ^{2a_{2}}, \\
\left[ a_{1}^{\psi },a_{3}^{2}\right] ^{a_{2}^{\psi }} &=&\left[ a_{1}^{\psi
},a_{3}^{2}\right] ^{a_{2}} \\
\left[ a_{1}^{\psi },a_{3}^{2}\right] ^{\left[ a_{2},\psi \right] } &=&\left[
a_{1}^{\psi },a_{3}^{2}\right] \text{.}
\end{eqnarray*}

Suppose $A$ has odd order. Then, $\left[ a_{1}^{\psi },a_{3}\right] ^{\left[
a_{2},\psi \right] }=\left[ a_{1}^{\psi },a_{3}\right] $ and therefore $\xi =%
\left[ a_{1}^{\psi }a_{2}^{\psi }a_{3}^{\psi },a_{1}a_{2}a_{3}\right] =e$.
By the previous corollary, we can substitute the $a_{i}$'s by their powers
in this last equation.
\end{proof}

For groups $A$ of odd order the reduction is drastic.

\begin{corollary}
Let $A$ be a finite abelian group of odd order generated by $S$ and $G=\chi
\left( A,S;2\right) $.\ Then, $G=\chi \left( A\right) $.
\end{corollary}

\begin{proof}
Let $\left\vert S\right\vert =m\geq 3$. We proceed by induction on $m$. By
the previous proposition, $\chi \left( A;2\right) =\chi \left( A;3\right) $.
We assume $\chi \left( A;2\right) =\chi \left( A;m-1\right) $. Then we
simply apply our argument to the set%
\begin{equation*}
\left\{ a_{1},a_{2},...,a_{m-2,}a_{m-1}a_{m}\right\}
\end{equation*}%
with $m-1$ elements and obtain%
\begin{equation*}
\left[ a_{1}^{\psi }a_{2}^{\psi }...a_{m-2}^{\psi }\left(
a_{m-1}a_{m}\right) ^{\psi },a_{1}a_{2}...a_{m-2}\left( a_{m-1}a_{m}\right) %
\right] =e\text{.}
\end{equation*}
\end{proof}

\textbf{Example 1}. The following example provides us with a glimpse into
the problem of reduction of the presentation of $G\left( A_{p,k},f\right) $
in general and how it compares with that of $\chi \left( A_{p,k}\right) $.

Let $A,B$ be isomorphic to $A_{p,3}$ with respective generators $\left\{
a_{1},a_{2},a_{3}\right\} ,\left\{ b_{1},b_{2},b_{3}\right\} $. Define%
\begin{equation*}
G=\,\left\langle 
\begin{array}{c}
A,B\mid \left[ a_{i},b_{i}\right] =e\text{ }\left( i=1,2,3\right) , \\ 
\left[ a_{1}a_{2},b_{1}b_{2}^{-1}\right] =\left[ a_{1}a_{3},b_{1}b_{3}\right]
=\left[ a_{2}a_{3},b_{2}b_{3}\right] =e\text{. }%
\end{array}%
\right\rangle \text{.}
\end{equation*}%
With the use of GAP, we find that the resulting group for $p=3,5,7$ to be
finite metabelian of order $p^{11}$ and of nilpotency class $3$. We also
find that 
\begin{equation*}
\left[ a_{1}a_{2}^{-1},b_{1}b_{2}\right] =\left[
a_{1}a_{3}^{-1},b_{1}b_{3}^{-1}\right] =\left[
a_{2}a_{3}^{-1},b_{2}b_{3}^{-1}\right] =e
\end{equation*}%
hold but $\left[ a_{1}a_{2}a_{3},b_{1}^{i}b_{2}^{j}b_{3}\right] \not=e$ for
any $1\leq i,j\leq p-1$. These results should be compared with those for $%
\chi \left( A_{p,3}\right) $ which has order $p^{9}$ and nilpotency class $2$%
.

We go back to the case $\chi \left( A_{p,3}\right) $ for $p=2$.

\begin{theorem}
Let $A_{2,3}$ be generated by $S=\left\{ a_{1},a_{2},a_{3}\right\} $, $%
G=\chi \left( A_{2,3},S;2\right) $ and $\xi =\left[ a_{1}^{\psi }a_{2}^{\psi
}a_{3}^{\psi },a_{1}a_{2}a_{3}\right] $. Then the kernel $K$ of the
epimorphism $\phi :G\rightarrow \chi \left( A_{2,3}\right) $ extended from $%
a_{i}\rightarrow a_{i},a_{i}^{\psi }\rightarrow a_{i}^{\psi }$ $\left(
i=1,2,3\right) $ is the normal closure of $\left\langle \xi \right\rangle $
in $G$ and is free abelian of rank $4$.
\end{theorem}

\begin{proof}
We will show that $K$ is freely generated by 
\begin{equation*}
\left\{ \xi ,\xi ^{a_{i}}\text{ }\left( i=1,2,3\right) \right\}
\end{equation*}%
and that $G$ acts on it as follows: for $\left\{ i,j,k\right\} =\left\{
1,2,3\right\} $, 
\begin{equation*}
\xi ^{\psi }=\xi ^{-1},\xi ^{a_{i}^{\psi }}=\xi ^{-a_{i}},\xi
^{a_{i}a_{j}}=\xi ^{-a_{k}}=\xi ^{a_{j}a_{i}^{\psi }}\text{.}
\end{equation*}%
We sketch the proof. First, we derive the table%
\begin{eqnarray*}
\left[ a_{3}^{\psi },a_{2},a_{1}\right] ^{a_{2}^{\psi }} &=&\left[
a_{2}^{\psi },a_{1},a_{3}\right] ^{-1},\left[ a_{3}^{\psi },a_{2},a_{1}%
\right] ^{a_{3}^{\psi }}=\left[ a_{1}^{\psi },a_{3},a_{2}\right] ^{-1}, \\
\left[ a_{3},a_{2}^{\psi },a_{1}^{\psi }\right] ^{a_{2}} &=&\left[
a_{2},a_{1}^{\psi },a_{3}^{\psi }\right] ^{-1},\left[ a_{3},a_{2}^{\psi
},a_{1}^{\psi }\right] ^{a_{3}}=\left[ a_{1},a_{3}^{\psi },a_{2}^{\psi }%
\right] ^{-1}, \\
\left[ a_{3}^{\psi },a_{2},a_{1}^{\psi }\right] ^{a_{2}} &=&\left[
a_{2}^{\psi },a_{1},a_{3}^{\psi }\right] ^{-1},\left[ a_{3}^{\psi
},a_{2},a_{1}^{\psi }\right] ^{a_{3}}=\left[ a_{1}^{\psi },a_{3},a_{2}^{\psi
}\right] ^{-1}, \\
\left[ a_{1}^{\psi },a_{2},a_{3}^{\psi }\right] ^{a_{1}} &=&\left[
a_{3}^{\psi },a_{1},a_{2}^{\psi }\right] ^{-1},\left[ a_{1}^{\psi
},a_{3},a_{2}^{\psi }\right] ^{a_{1}}=\left[ a_{2}^{\psi },a_{1},a_{3}^{\psi
}\right] ^{-1}, \\
\left[ a_{2}^{\psi },a_{1},a_{3}^{\psi }\right] ^{a_{2}} &=&\left[
a_{3}^{\psi },a_{2},a_{1}^{\psi }\right] ^{-1}\text{.}
\end{eqnarray*}

From this table we conclude that the subgroup generated by 
\begin{equation*}
\left\{ \left[ a_{3}^{\psi },a_{2},a_{1}\right] ,\left[ a_{1}^{\psi
},a_{3},a_{2}\right] ,\left[ a_{2}^{\psi },a_{1},a_{3}\right] ,\left[
a_{3}^{\psi },a_{2},a_{1}\right] ^{a_{1}^{\psi }}\right\}
\end{equation*}%
is abelian and normal in $G$.

Next, we find%
\begin{eqnarray*}
\xi  &=&\left[ a_{1}^{\psi }a_{2}^{\psi }a_{3}^{\psi },a_{1}a_{2}a_{3}\right]
=\left[ a_{2}^{\psi },\left[ a_{3}^{\psi },a_{1}\right] \right] \left[
a_{3}^{\psi },a_{1},a_{2}\right]  \\
&=&\left[ a_{2}^{\psi },\left[ a_{1}^{\psi },a_{3}\right] \right] \left[
a_{1}^{\psi },a_{3},a_{2}\right] 
\end{eqnarray*}%
and by permuting the $a_{i}$'s, the following holds%
\begin{eqnarray*}
\xi  &=&\left[ a_{3}^{\psi },\left[ a_{2}^{\psi },a_{1}\right] \right] \left[
a_{2}^{\psi },a_{1},a_{3}\right]  \\
&=&\left[ a_{1}^{\psi },\left[ a_{3}^{\psi },a_{2}\right] \right] \left[
a_{3}^{\psi },a_{2},a_{1}\right]  \\
&=&\left[ a_{1}^{\psi },\left[ a_{2}^{\psi },a_{3}\right] \right] \left[
a_{2}^{\psi },a_{3},a_{1}\right] \text{.}
\end{eqnarray*}%
It is straightforward to obtain the action of $G$ on $\left\{ \xi ,\xi
^{a_{i}}\text{ }\left( i=1,2,3\right) \right\} $, as described above.

Let $\mathbb{Z}\left[ x,y,z\right] $ be the polynomial ring in the variables 
$x,y,z$ with coefficients from $\mathbb{Z}$. The proof is finished by
constructing the group as a subgroup of $GL\left( 5,\mathbb{Z}\left[ x,y,z%
\right] \right) $:%
\begin{eqnarray*}
a_{1} &\rightarrow &\left( 
\begin{array}{ccccc}
0 & 1 & 0 & 0 & 0 \\ 
1 & 0 & 0 & 0 & 0 \\ 
0 & 0 & 0 & -1 & 0 \\ 
0 & 0 & -1 & 0 & 0 \\ 
x & -x & y & y & 1%
\end{array}%
\right) ,a_{2}\rightarrow \left( 
\begin{array}{ccccc}
0 & 0 & 1 & 0 & 0 \\ 
0 & 0 & 0 & -1 & 0 \\ 
1 & 0 & 0 & 0 & 0 \\ 
0 & -1 & 0 & 0 & 0 \\ 
x & y & -x & y & 1%
\end{array}%
\right) , \\
a_{3} &\rightarrow &\left( 
\begin{array}{ccccc}
0 & 0 & 0 & 1 & 0 \\ 
0 & 0 & -1 & 0 & 0 \\ 
0 & -1 & 0 & 0 & 0 \\ 
1 & 0 & 0 & 0 & 0 \\ 
x & y & y & -x & 1%
\end{array}%
\right) ,a_{1}^{\psi }\rightarrow \left( 
\begin{array}{ccccc}
0 & -1 & 0 & 0 & 0 \\ 
-1 & 0 & 0 & 0 & 0 \\ 
0 & 0 & 0 & -1 & 0 \\ 
0 & 0 & -1 & 0 & 0 \\ 
w & w & y & y & 1%
\end{array}%
\right) , \\
a_{2}^{\psi } &\rightarrow &\left( 
\begin{array}{ccccc}
0 & 0 & -1 & 0 & 0 \\ 
0 & 0 & 0 & -1 & 0 \\ 
-1 & 0 & 0 & 0 & 0 \\ 
0 & -1 & 0 & 0 & 0 \\ 
w & y & w & y & 1%
\end{array}%
\right) ,a_{3}^{\psi }\rightarrow \left( 
\begin{array}{ccccc}
0 & 0 & 0 & -1 & 0 \\ 
0 & 0 & -1 & 0 & 0 \\ 
0 & -1 & 0 & 0 & 0 \\ 
-1 & 0 & 0 & 0 & 0 \\ 
w & y & y & w & 1%
\end{array}%
\right)
\end{eqnarray*}

from which we find that 
\begin{equation*}
\xi \rightarrow \left( 
\begin{array}{ccccc}
1 & 0 & 0 & 0 & 0 \\ 
0 & 1 & 0 & 0 & 0 \\ 
0 & 0 & 1 & 0 & 0 \\ 
0 & 0 & 0 & 1 & 0 \\ 
z & 0 & 0 & 0 & 1%
\end{array}%
\right)
\end{equation*}

where $z=4\left( -x-y+w\right) $.
\end{proof}

\section{\textbf{Fixing a basis}}

Let\ $A=A_{p,k}$ be written additively and let $\mathcal{B}_{k}$\textit{\ }%
be the set of bases of\textit{\ }$A$. We will show that for any bijection $%
f:A\rightarrow A$ fixing $0$, the set $\mathcal{B}_{k}\cap \mathcal{B}%
_{k}^{f}$ is nonempty; indeed, $lim_{p\rightarrow \infty }$ $\frac{%
\left\vert \mathcal{B}_{k}\cap \mathcal{B}_{k}^{f}\right\vert }{\left\vert 
\mathcal{B}_{k}\right\vert }=1$.

\textbf{Example 2. (i)} The following easy example shows that for $p$ odd, a
bijection $f:A\rightarrow A$ may be linear when restricted to each of the $1$%
-dimensional subspaces and may also permute $\mathcal{B}_{k}$, without being
a linear transformation. Let $A=A_{p,2}$ be generated by $a_{1},a_{2}$ and
define $f:A\rightarrow A$ by $f:ia_{1}\rightarrow ia_{1},ia_{2}\rightarrow
ia_{2},i\left( a_{1}+ja_{2}\right) \rightarrow i\left( a_{1}-ja_{2}\right) $
for $0\leq i,j\leq p-1,j\not=0$.

\textbf{(ii) }A map $f$ is \textit{anti-additive }provided\textit{\ }$\left(
x+y\right) ^{f}\not=x^{f}+y^{f}$ for all $x,y$ such that $0\not\in \left\{
x,y,x+y\right\} $. Let $F$ be a field of characteristic different from $3$
and such that its multiplicative group $F^{\#}$ does not contain elements of
order $3$. Then, multiplicative inversion defined by $f:0\rightarrow
0,x\rightarrow x^{-1}$ is anti-additive.

Concerning the first example, the situation for $k\geq 3$ is quite different
as can be seen from a result of R. Baer from 1939 (\cite{Suz}, Th. 2, page
35):

\textit{let }$G$\textit{\ be an abelian }$p$\textit{-group such that }$G$%
\textit{\ contains an element of order }$p^{n}$\textit{\ and contains at
least }$3$\textit{\ independent elements of such order. Then any
projectivity of }$G$\textit{\ onto another abelian group }$H$\textit{\ is
induced by an isomorphism}.

It follows then

\begin{lemma}
Let\textit{\ }$A=A_{p,k}$. \textit{Suppose }$f$ is a permutation of $A^{\#}$%
. If $f$ \textit{permutes the set }$\mathcal{B}_{k}$\textit{\ of bases of }$%
A $ then $f$\textit{\ is a linear transformation when }$k\geq 3$ and when $%
p=2,k=2$.
\end{lemma}

\begin{proof}
As $f$ permutes the set $B_{k}$\ of bases of $A$, it induces a projectivity
on $A$. The case $p=2,k=2$ is easy.
\end{proof}

\begin{definition}
Let $A=A_{p,k}$ and $f$ be a permutation of $A^{\#}$. A subset $C$ of $A$ of
linearly independent elements is said to be $f$\textit{-independent} if $%
C^{f}$ is also a linearly independent set. Let $N(p,j)$ denote the number of 
$f$\textit{-independent }subsets $C$ of $A$ with $\left\vert C\right\vert =j$%
.
\end{definition}

\begin{proposition}
Maintain the previous notation. Then, the following inequality holds for all 
$k\geq 1$,%
\begin{equation*}
N\left( p,k\right) \geq \left( p^{k}-1\right) \dprod\limits_{1\leq i\leq k-1}%
\frac{p^{k-j}-1}{p-1}\left( p^{j}-2p^{j-1}+j+1\right) \text{.}
\end{equation*}%
Furthermore, $lim_{p\rightarrow \infty }$ $\frac{\left\vert \mathcal{B}%
_{k}\cap \mathcal{B}_{k}^{f}\right\vert }{\left\vert \mathcal{B}%
_{k}\right\vert }=1$.
\end{proposition}

\begin{proof}
Clearly, $N\left( p,1\right) =p^{k}-1$. Let $k\geq 2$, $C$ be $f$%
-independent and $\left\vert C\right\vert =j\geq 1$. Denote $U=\left\langle
C\right\rangle ,W=\left\langle C^{f}\right\rangle $. Then, 
\begin{eqnarray*}
\left\vert U\right\vert &=&\left\vert W\right\vert =p^{j}, \\
\left\vert U^{\#}-C\right\vert &=&\left\vert W^{\#}-C^{f}\right\vert =\left(
p^{j}-1\right) -j\text{.}
\end{eqnarray*}%
Suppose $U\not=A$. There are $\frac{p^{k-j}-1}{p-1}$ non-trivial cyclic
subgroups in the quotient group $\frac{A}{U}$. For each such cyclic
subgroup, choose a representative $P$ in $A$ and also choose a generator $v$
for each $P$.

Fix such a $P=$ $\left\langle v\right\rangle $. Then, each element in the
set 
\begin{equation*}
L=\left\{ u+iv\mid u\in U,1\leq i\leq p-1\right\}
\end{equation*}%
is independent of $C$ and $\left\vert L\right\vert =\left\vert
L^{f}\right\vert =p^{j}\left( p-1\right) $. The elements of $L^{f}\backslash
W^{\#}=L^{f}\backslash \left( W^{\#}-C^{f}\right) $ are independent of $%
C^{f} $ and 
\begin{eqnarray*}
\left\vert L^{f}\backslash W^{\#}\right\vert &\geq &p^{j}\left( p-1\right)
-\left( p^{j}-1-j\right) \\
&=&p^{j+1}-2p^{j}+j+1\text{.}
\end{eqnarray*}

Therefore, 
\begin{equation*}
N\left( p,j+1\right) \geq N\left( p,j\right) \frac{p^{k-j}-1}{p-1}\left(
p^{j+1}-2p^{j}+j+1\right)
\end{equation*}%
and 
\begin{equation*}
N\left( p,k\right) \geq \left( p^{k}-1\right) \dprod\limits_{1\leq j\leq k-1}%
\frac{p^{k-j}-1}{p-1}\left( p^{j+1}-2p^{j}+j+1\right) \text{.}
\end{equation*}

Finally, since 
\begin{eqnarray*}
\left\vert \mathcal{B}_{k}\right\vert &=&\dprod\limits_{0\leq j\leq
k-1}\left( p^{k}-p^{j}\right) =p^{\binom{k}{2}}\dprod\limits_{0\leq j\leq
k-1}\left( p^{k-j}-1\right) , \\
N\left( p,k\right) &=&\left\vert \mathcal{B}_{k}\cap \mathcal{B}%
_{k}^{f}\right\vert \text{,}
\end{eqnarray*}
we conclude 
\begin{equation*}
\frac{\left\vert \mathcal{B}_{k}\cap \mathcal{B}_{k}^{f}\right\vert }{%
\left\vert \mathcal{B}_{k}\right\vert }\geq \frac{\dprod\limits_{1\leq j\leq
k-1}\left( p^{j+1}-2p^{j}+j+1\right) }{\left( p-1\right) ^{k-1}p^{\binom{k}{2%
}}}
\end{equation*}
and $lim_{p\rightarrow \infty }$ $\frac{\left\vert \mathcal{B}_{k}\cap 
\mathcal{B}_{k}^{f}\right\vert }{\left\vert \mathcal{B}_{k}\right\vert }=1$
follows easily.
\end{proof}

\begin{corollary}
Let $f:A_{p,k}^{\#}\rightarrow A_{p,k}^{\#}$ be a bijection. Then there
exists $g$ an element in the double coset $GL(k,p).f.GL(k,p)$ and there
exists a basis $C$ of $A_{p,k}^{\#}$ such that $g$ fixes point-wise the
elements of $C$.
\end{corollary}

\section{Permuting the set of cyclic subgroups}

The commutation between two elements in a group imply the commutation of the
cyclic groups generated by them. For this reason, it is important to
consider commutation correspondence between cyclic subgroups.

Let $A=$ $A_{p,k}$ and let $f$ be a permutation of $A^{\#}$. Define $%
R\subset A_{p,k}^{\#}\times A_{p,k}^{\#}$ by%
\begin{equation*}
R=\left\{ (a^{i},b^{j})\text{ }\left( 1\leq i,j\leq p-1\right) \mid
a^{f}=b\right\} \text{.}
\end{equation*}%
We will prove that $R$ contains at least $p-1$ permutations $g$ of $%
A_{p,k}^{\#}$ such that $\left( a^{i}\right) ^{g}=\left( a^{g}\right) ^{i}$
for all $1\leq i\leq p-1$. Therefore, $G\left( A_{p,k};f\right) $ is a
quotient of $G\left( A_{p,k};g\right) $ for each one of these $g$'s. For
this purpose, we construct a multi-edge digraph $L$ from $R$, having
vertices the non-trivial cyclic subgroups $C_{i}$ of $A$ and edges $\left(
C,C^{\prime }\right) $ whenever $C=\left\langle a^{i}\right\rangle
,C^{\prime }=\left\langle \left( a^{\prime }\right) ^{j}\right\rangle $ and $%
f:a^{i}\rightarrow \left( a^{\prime }\right) ^{j}$. Then $L$ is a regular
graph, in the sense that there are exactly $p-1$ edges coming into and $p-1$
edges leaving each vertex.

We enumerate the vertices of $L$ and let $N=\left( N_{ij}\right) $ be the
incidence matrix with respect to this enumeration; that is $N_{ij}=l$ if and
only there are a total of $l$ edges connecting the vertex $i$ to the vertex $%
j$. Then $N$ is doubly stochastic, as all row and column sums of $N$ are
equal to $p-1$. A permutation $g$ contained in $R$ corresponds to a non-zero
monomial $N_{1,1^{\sigma }}N_{2,2^{\sigma }}...N_{k,k^{\sigma }}$ for some
permutation $\sigma $ of $\left\{ 1,2,...,k\right\} $.

\begin{definition}
Let $M=\left( M_{ij}\right) ,N=\left( N_{ij}\right) $ be $k\times k$
matrices over the real numbers. Then, (i) $M,N$ are equivalent provided
there exist permutational matrices $S,T$ such that $M=$ $SNT$; (ii) $N$ is
said to be totally singular provided 
\begin{equation*}
N_{1,1^{\sigma }}N_{2,2^{\sigma }}...N_{k,k^{\sigma }}=0
\end{equation*}%
for all permutations $\sigma $ of $\left\{ 1,2,...,k\right\} $.
\end{definition}

\begin{proposition}
Let $N$ be a totally singular $k\times k$ matrix over the real numbers. Then 
$N$ is equivalent to a matrix which contains a submatrix $0_{\left(
k-l\right) \times \left( l+1\right) }$ for \ some $l\geq 0$.
\end{proposition}

\begin{proof}
By induction on $k$. The cases $k=2,3$ are easy; that is, if $k=2$ then $N$
is equivalent to $\left( 
\begin{array}{cc}
\ast & 0 \\ 
\ast & 0%
\end{array}%
\right) $ and the if $k=3$ then $N$ is equivalent to one of%
\begin{equation*}
\left( 
\begin{array}{ccc}
\ast & \ast & 0 \\ 
\ast & \ast & 0 \\ 
\ast & \ast & 0%
\end{array}%
\right) ,\left( 
\begin{array}{ccc}
\ast & \ast & \ast \\ 
\ast & 0 & 0 \\ 
\ast & 0 & 0%
\end{array}%
\right) \text{.}
\end{equation*}

Suppose that the assertion is true for $k$. We consider $N$ of dimension $%
k+1 $. Then, we can assume that there exist an $l\geq 0$ such that 
\begin{equation*}
N=\left( 
\begin{array}{ccc}
a_{11} & .. & U_{1\times \left( l+1\right) } \\ 
V_{l\times 1} & B_{l\times \left( k-l-1\right) } & C_{l\times \left(
l+1\right) } \\ 
W_{\left( k-l\right) \times 1} & D_{\left( k-l\right) \times \left(
k-l-1\right) } & 0_{\left( k-l\right) \times \left( l+1\right) }%
\end{array}%
\right) \text{.}
\end{equation*}%
If $U_{1\times \left( l+1\right) }$ or any row of $C_{l\times \left(
l+1\right) }$ is null then we obtain the desired form. We can also assume
that $U_{1\times \left( l+1\right) }=\left( ...,a_{1k}\right) $, $a_{1k}>0$.
Therefore, we have the $\left( l+1\right) \times \left( l+1\right) $ matrix%
\begin{equation*}
\left( 
\begin{array}{c}
U_{1\times \left( l+1\right) } \\ 
C_{l\times \left( l+1\right) }%
\end{array}%
\right) =\left( 
\begin{array}{cc}
... & a_{1k} \\ 
R_{l\times l} & S_{l\times 1}%
\end{array}%
\right) \text{.}
\end{equation*}%
We have%
\begin{equation*}
Y_{\left( k-l\right) \times \left( k-l\right) }=\left( 
\begin{array}{cc}
W_{\left( k-l\right) \times 1} & D_{\left( k-l\right) \times \left(
k-l-1\right) }%
\end{array}%
\right)
\end{equation*}%
and therefore 
\begin{equation*}
N=\left( 
\begin{array}{cc}
\ast & Z_{\left( l+1\right) \times \left( l+1\right) } \\ 
Y_{\left( k-l\right) \times \left( k-l\right) } & 0_{\left( k-l\right)
\times \left( l+1\right) }%
\end{array}%
\right) \text{.}
\end{equation*}%
Now, one of $Y_{\left( k-l\right) \times \left( k-l\right) },Z_{\left(
l+1\right) \times \left( l+1\right) }$ is totally singular; suppose it is
the first one. Then we may assume 
\begin{equation*}
Y_{\left( k-l\right) \times \left( k-l\right) }=\left( 
\begin{array}{cc}
... & Y_{m\times \left( m+1\right) }^{\prime } \\ 
... & 0_{\left( k-l-m\right) \times \left( m+1\right) }%
\end{array}%
\right) \text{.}
\end{equation*}%
Hence 
\begin{equation*}
N=\left( 
\begin{array}{cc}
\ast & Z_{\left( l+1\right) \times \left( l+1\right) } \\ 
\begin{array}{cc}
... & Y_{m\times \left( m+1\right) }^{\prime } \\ 
... & 0_{\left( k-l-m\right) \times \left( m+1\right) }%
\end{array}
& 0_{\left( k-l\right) \times \left( l+1\right) }%
\end{array}%
\right)
\end{equation*}%
and we obtain in $N$ a $\left( k-l-m\right) \times \left( m+1+l+1\right) $
block of zeroes where the sum of the dimensions is $k+1$.
\end{proof}

\begin{corollary}
Maintain the previous notation. Suppose the entries of $N$ are non-negative.
If in addition $N$ is doubly stochastic then $N=0$.
\end{corollary}

\begin{proof}
Let the row sum be $s$. There exists $l\geq 0$ such that 
\begin{equation*}
N=\left( 
\begin{array}{cc}
X_{l\times \left( k-l\right) } & Z_{l\times \left( l+1\right) } \\ 
Y_{\left( k-l\right) \times \left( k-l\right) } & 0_{\left( k-l\right)
\times \left( l+1\right) }%
\end{array}%
\right) \text{.}
\end{equation*}%
Therefore the column sum of $\left( 
\begin{array}{c}
Z_{l\times \left( l+1\right) } \\ 
0_{\left( k-l\right) \times \left( l+1\right) }%
\end{array}%
\right) $ is $\left( l+1\right) s$ whereas the row sum is at most $ls$;
hence $s=0$.
\end{proof}

We go back to our graph $L$ and its incidence matrix $N$ which is doubly
stochastic with $s=p-1$. Then there exists a monomial $N_{1,1^{\sigma
}}N_{2,2^{\sigma }}...N_{k,k^{\sigma }}\not=0$ and so $N_{i,i^{\sigma
}}\not=0$ for all $i$. This produces for us a bijection $g:A_{p,k}^{\#}%
\rightarrow A_{p,k}^{\#}$. By removing the edges corresponding to $g$, the
graph $L$ is reduced to one which is $\left( s-1\right) $-regular.
Therefore, we can produce in this manner $p-1$ permutations $g$. Clearly, if 
$f$ is an isomorphism on the cyclic subgroups then all the permutations $g$
are equal.

\section{Classification of $G\left( A;f\right) $ for $A$ of small rank}

We treat in this section groups $G\left( A,B;f\right) $ were $A,B$ are
finite abelian groups generated by at most $4$ elements.

\begin{proposition}
Suppose $A=\left\langle a\right\rangle ,B=\left\langle b\right\rangle $ are
cyclic groups having equal finite orders $n$. Then, $G=G\left( A,B;f\right) $
is isomorphic to $A\times B$.
\end{proposition}

\begin{proof}
Suppose $G$ is not abelian. Let $1<r,s$ $<n$ be minimal integers such that $%
\left[ a,b^{r}\right] =e=\left[ a^{s},b\right] $. Then, 
\begin{eqnarray*}
f &:&\left\{ a^{i}\mid 1<i<n,\gcd \left( i,s\right) =1\right\} \rightarrow \\
&&\left\{ b^{j}\mid 1\leq j<n,r|j\right\} , \\
\phi \left( s\right) \frac{n}{s} &<&\frac{n}{r}; \\
f^{-1} &:&\left\{ b^{i}\mid 1<i<n,\gcd \left( i,r\right) =1\right\} \\
&\rightarrow &\left\{ a^{j}\mid 1\leq j<n,s|j\right\} , \\
\phi \left( r\right) \frac{n}{r} &<&\frac{n}{s}; \\
\phi \left( s\right) \phi \left( r\right) \frac{n}{r} &<&\phi \left(
s\right) \frac{n}{s}<\frac{n}{r}
\end{eqnarray*}%
which is a contradiction.
\end{proof}

\textbf{Example 4.} The following example shows that relaxing $f$ from
bijection to surjection may not maintain the finiteness of $G\left(
A,B;f\right) $.

Let $A=\left\langle a\right\rangle $ de a cyclic group of order $p^{3}$, $%
B=\left\langle b\right\rangle $ be cyclic of order $p^{2}$ and define $%
f:A\rightarrow B$ by choosing surjective maps 
\begin{eqnarray*}
f &:&e\rightarrow e, \\
A\backslash \left\langle a^{p}\right\rangle &\rightarrow &\left\langle
b^{p}\right\rangle \backslash \left\{ e\right\} , \\
\left\langle a^{p}\right\rangle \backslash \left\{ e\right\} &\rightarrow
&B\backslash \left\langle b^{p}\right\rangle \text{.}
\end{eqnarray*}%
Then the relations $\left[ a,a^{f}\right] =e$ in $G\left( A,B;f\right) $ are
equivalent to $\left\langle a^{p},b^{p}\right\rangle $ being central in $G$.
Therefore, $\frac{G}{\left\langle a^{p},b^{p}\right\rangle }$ is isomorphic
to the free product $C_{p}\ast C_{p}$.

\begin{proposition}
Let $A=\left\langle a_{1},a_{2}\right\rangle ,B=$ $\left\langle
b_{1},b_{2}\right\rangle $ be homogenous abelian groups of rank $2$, both
having finite exponent $n$. Then $G=G(A,B;f)$ is nilpotent of class at most $%
2$ and its derived subgroup is cyclic of order divisor of $n$.
\end{proposition}

\begin{proof}
Let us call an element of $A$ which is part of some $2$-generating set of $A$
\textit{primitive}. The non-primitive elements are of the form is $%
a_{1}^{i}a_{2}^{j}$\ where $\gcd \left( i,n\right) \not=1\not=\gcd \left(
j,n\right) $; therefore, their number is $\left( n-\varphi \left( n\right)
\right) ^{2}$. The number of primitive elements is $n^{2}-\left( n-\varphi
\left( n\right) \right) ^{2}=2n\varphi \left( n\right) -\varphi \left(
n\right) ^{2}$.

The difference between the number of primitive elements and the
non-primitives is positive: 
\begin{eqnarray*}
&&2n\varphi \left( n\right) -\varphi \left( n\right) ^{2}-\left( n-\varphi
\left( n\right) \right) ^{2} \\
&=&4n\varphi \left( n\right) -2\varphi \left( n\right) ^{2}-n^{2}
\end{eqnarray*}%
and

\begin{equation*}
2\frac{\varphi \left( n\right) }{n}+2\frac{\varphi \left( n\right) }{n}\geq
1+2\left( \frac{\varphi \left( n\right) }{n}\right) ^{2}\text{,}
\end{equation*}%
since $\frac{1}{2}\leq \frac{\varphi \left( n\right) }{n}<1$.

Since $f$ is a bijection we may suppose $f:a_{1}\rightarrow b_{1}$. Now, any 
$\ 2$-generating set of $A$ containing $a_{1}$ has the form $\left\{
a_{1},a_{1}^{l}a_{2}^{m}\right\} $ where $\gcd \left( m,n\right) =1$; there
are $n\varphi \left( n\right) $ such elements $a_{1}^{l}a_{2}^{m}$. As $%
n\varphi \left( n\right) >\left( n-\varphi \left( n\right) \right) ^{2}$, we
may suppose $f:a_{2}\rightarrow b_{2}$.

As $f:a_{1}^{i}a_{2}^{j}\rightarrow b_{1}^{k}b_{2}^{l}$, we have%
\begin{equation*}
\left[ a_{1}^{i}a_{2}^{j},b_{1}^{k}b_{2}^{l}\right] =\left[
a_{1}^{i},b_{2}^{l}\right] \left[ a_{2}^{j},b_{1}^{k}\right] =e\text{.}
\end{equation*}%
Since $a_{1},b_{2}$ commute with $\left[ a_{2}^{j},b_{1}^{k}\right] $, while 
$a_{2},b_{1}$ commute with $\left[ a_{1}^{i},b_{2}^{l}\right] $, we conclude
that $\left[ b_{1}^{k},a_{2}^{j}\right] =\left[ a_{1}^{i},b_{2}^{l}\right] $
is central. We note that the size of 
\begin{equation*}
\left\{ a_{1}^{i}a_{2}^{j}\mid \gcd \left( i,n\right) =1,j\not=0\right\}
\end{equation*}%
$\varphi \left( n\right) \left( n-1\right) $,whereas the size of 
\begin{equation*}
\left\{ b_{1}^{k}b_{2}^{l}\mid k\not=0,\gcd \left( l,n\right)
\not=1,l\not=0\right\}
\end{equation*}
is $\left( n-1\right) \left( n-\varphi \left( n\right) -1\right) $.

As the first set is larger than the second, there exist $i,j$ with $\gcd
\left( i,n\right) =1,j\not=0$ such that $f:a_{1}^{i}a_{2}^{j}\rightarrow
b_{1}^{k}b_{2}^{l}$ where $\gcd \left( l,n\right) =1$. We rewrite $a_{1}^{i}$
as $a_{1}$ and $b_{2}^{l}$ as $b_{2}$ and conclude that $\left[ a_{1},b_{2}%
\right] $ is central. Similarly, $\left[ a_{2},b_{1}\right] $ is also
central.

Hence,%
\begin{equation*}
\left[ a_{1},b_{2}\right] ^{n}=\left[ a_{2},b_{1}\right] ^{n}=e
\end{equation*}%
and there exist $0\leq s,t\leq n-1$ such that 
\begin{equation*}
\left[ a_{1},b_{2}\right] =\left[ a_{2},b_{1}\right] ^{s},\left[ a_{2},b_{1}%
\right] =\left[ a_{1},b_{2}\right] ^{t}
\end{equation*}

and $G^{\prime }=\left\langle \left[ a_{1},b_{2}\right] \right\rangle
=\left\langle \left[ a_{2},b_{1}\right] \right\rangle $.
\end{proof}

\subsection{The groups $G\left( A_{p,k};f\right) $ for $p=2,3$ and $k=$ $3,4$%
}

We consider the groups 
\begin{equation*}
G\left( A_{2,3};f\right) ,G\left( A_{2,4};f\right) ,G\left( A_{3,3};f\right) 
\text{,}
\end{equation*}%
their orders, nilpotency classes $c$ and derived lengths $d$. We write the
group $A_{p,k}$ additively.

(i) The group $A_{2,3}$ is 
\begin{equation*}
\left\{
0,a_{1},a_{2},a_{3},a_{1}+a_{2},a_{1}+a_{3},a_{2}+a_{3},a_{1}+a_{2}+a_{3}%
\right\} \text{,}
\end{equation*}%
which we enumerate lexicographically and identify its elements with their
positions in this order. The group $SL(3,2)$ in its linear action on $%
A_{2,3}^{\#}$, is generated by the permutations $\left( 2,7,4,6,5,8,3\right)
,\left( 2,8,7\right) \left( 3,4,6\right) $. Using GAP, we find that there
are $4$ double cosets in $SL(3,2)\backslash Sym\left( 7\right) /SL(3,2)$,
which are represented by the permutations%
\begin{equation*}
\left\{ \left( {}\right) ,\left( 6,7\right) ,\left( 6,7,8\right) ,\left(
5,6,7,8\right) \right\} \text{.}
\end{equation*}%
\newline
Each permutation produces for us a bijection $f$ and a group as in the table
below%
\begin{equation*}
\begin{tabular}{|c|c|c|c|}
\hline
$f$ & $\left\vert G\left( A_{2,3};f\right) \right\vert $ & $c$ & $d$ \\ 
\hline
$()$ & $2^{10}$ & $3$ & $2$ \\ \hline
$(6,7)$ & $2^{10}$ & $3$ & $2$ \\ \hline
$(6,7,8)$ & $2^{8}$ & $2$ & $2$ \\ \hline
$(5,6,7,8)$ & $2^{8}$ & $2$ & $2$ \\ \hline
\end{tabular}%
\text{.}
\end{equation*}%
Further analysis shows that these $4$ groups are non-isomorphic.

(ii) The group $A_{2,4}$ is treated in a similar manner. We find that there
are $3374$ double cosets in $SL(4,2)\backslash Sym\left( 15\right) /SL(4,2)$%
. The corresponding groups $G\left( A_{2,4};f\right) $ have orders 
\begin{equation*}
2^{9},2^{10},2^{11},2^{12},2^{13},2^{15},2^{19}\text{.}
\end{equation*}%
There are $5$ representatives $f$ for which the groups have maximum order.
We list them below with their invariants $c,d$:%
\begin{equation*}
\begin{tabular}{|c|c|c|c|}
\hline
$f$ & $\left\vert G\left( A_{2,4};f\right) \right\vert $ & $c$ & $d$ \\ 
\hline
$()$ & $2^{19}$ & $4$ & $2$ \\ \hline
$(15,16)$ & $2^{19}$ & $3$ & $3$ \\ \hline
$(11,14)(15,16)$ & $2^{19}$ & $5$ & $3$ \\ \hline
$(9,11)(10,13)(12,14)$ & $2^{19}$ & $5$ & $3$ \\ \hline
$(9,12)(10,13)(11,14)$ & $2^{19}$ & $4$ & $2$ \\ \hline
\end{tabular}%
\text{.}
\end{equation*}%
Further analysis shows that these $5$ groups are non-isomorphic.

(iii) The set of $1$-dimensional subspaces of $A_{3,3}$ has size $13$. The
group $GL(3,3)$ in its linear action on $A_{3,3}$ induces the group $PGL(3,3)
$. There are $252$ double cosets in $PGL(3,3)\backslash Sym\left( 13\right)
/PGL(3,3)$. A double coset representative corresponds to a bijection $g$ of $%
A_{3,3}^{\#}$ which is linear on the $1$-dimensional subspaces of $A_{3,3}$.
We find that the corresponding groups have orders $3^{6},3^{7},3^{8},3^{9}$.
The groups of order $3^{6}$ are clearly isomorphic to $A_{3,3}\times A_{3,3}$%
. Those of higher order have nilpotency class $2$. There is a unique group
of maximum order $3^{9}$ which clearly is isomorphic to $\chi \left(
A_{3,3}\right) $.

\section{Three General Examples}

Working by hand with $G=G\left( H,K;f\right) $, it is easy to produce many
consequences from the defining relations: given $h\in H,k\in K$, then the
equalities 
\begin{equation*}
\left[ h,k\right] =\left[ h,h^{f}k\right] =\left[ k^{f^{-1}}h,k\right]
\end{equation*}%
hold and these serve to define the two maps 
\begin{equation*}
\alpha :\left( h,k\right) \rightarrow \left( h,h^{f}k\right) ,\beta :\left(
h,k\right) \rightarrow \left( k^{f^{-1}}h,k\right)
\end{equation*}%
on the set $H\times K$. Finding equivalent forms for $\left[ h,k\right] $
according to the above process corresponds to calculating the orbits of the
group $\left\langle \alpha ,\beta \right\rangle $ in its action on $H\times
K $. We will illustrate this sort of analysis in the examples below.

\subsubsection{The multiplicative inverse function}

Let $F,\dot{F}$ be isomorphic fields, via $a\rightarrow \dot{a}$. Define $%
f:0\rightarrow 0,k\rightarrow \dot{k}^{-1\text{.}}$ for $k\not=0$ and let 
\begin{equation*}
G=G\left( F,inv\right) =\left\langle F,\dot{F}\mid \left[ a,\frac{1}{\dot{a}}%
\right] =e\text{ for }a\not=0\right\rangle \text{.}
\end{equation*}

Given an integer $m$ , we have $\left[ a,\frac{m}{\dot{a}}\right] =e$.
Therefore, if $F=GF\left( p\right) $ or $\mathbb{Q}$, the group $G$ is
isomorphic to $F\times F$.

We will prove

\begin{theorem}
Let $F=GF\left( 2^{k}\right) $ where $2^{k}-1$ is a prime number. Then, $%
G\left( F;inv\right) $ is nilpotent of class at most $2$.
\end{theorem}

It appears that this result holds more generally. We develop below formulas
for general fields $F$.

\begin{lemma}
Let $i\geq 1,b\in F,b\not=0$ and suppose $(2i-1)!$ is invertible in $F$. Then%
\begin{eqnarray*}
\left( \alpha \beta \right) ^{i} &:&\left( 0,b\right) \rightarrow \left( 
\frac{(2i-1)!}{\left( 2^{i-1}(i-1)!\right) ^{2}}\frac{1}{b},\frac{%
2^{2i-2}\left( (i-1)!\right) ^{2}}{(2i-2)!}b\right) , \\
\left( \alpha \beta \right) ^{i}\alpha &:&\left( 0,b\right) \rightarrow
\left( \frac{(2i-1)!}{\left( 2^{i-1}(i-1)!\right) ^{2}}\frac{1}{b},\frac{%
2i\left( 2^{i-1}(i-1)!\right) ^{2}}{(2i-1)!}b\right) \text{.}
\end{eqnarray*}%
Let $p$ be an odd prime number. Then,%
\begin{equation*}
\left( \alpha \beta \right) ^{\frac{p+1}{2}}:\left( 0,b\right) \rightarrow
\left( 0,\left( -1\right) ^{\frac{p-1}{2}}b\right)
\end{equation*}%
modulo $p$.
\end{lemma}

\begin{proof}
The first formulae can be verified in a straightforward manner.

In case the characteristic of $F$ is a prime number $p$ then%
\begin{equation*}
\left( \alpha \beta \right) ^{\frac{p+1}{2}}:\left( 0,b\right) \rightarrow
\left( 0,\frac{2^{p-2}\left( \frac{p-1}{2}\right) !\left( \frac{p-3}{2}%
\right) !}{(p-2)!}b\right)
\end{equation*}

and by Wilson's theorem,%
\begin{eqnarray*}
(p-1)! &=&\left( -1^{2}\right) \left( -2^{2}\right) ...\left( -\left( \frac{%
p-1}{2}\right) ^{2}\right) \\
&=&\left( -1\right) ^{\frac{p-1}{2}}\left( \frac{p-1}{2}!\right) ^{2}\equiv
-1\text{ modulo }p\text{,} \\
\left( \frac{p-1}{2}!\right) ^{2} &\equiv &\left( -1\right) ^{\frac{p+1}{2}}%
\text{ modulo }p\text{.}
\end{eqnarray*}
\end{proof}

\begin{lemma}
Let $L=GF\left( p\right) $ when charac$\left( F\right) =p$ and $L=\mathbb{Z}$
when charac$\left( F\right) =0$. Let 
\begin{equation*}
T=\left\{ \left( a,b\right) \mid a\not=0\not=b,ab\not\in L\right\} \text{.}
\end{equation*}%
Then, $\alpha ,\beta :T\rightarrow T$ and for all integers $%
i_{1},...,i_{s},j_{1},...,j_{s}$,%
\begin{equation*}
\alpha ^{i_{1}}\beta ^{j_{1}}...\alpha ^{i_{s}}\beta ^{j_{s}}:\left(
a,b\right) \rightarrow \left( a^{\prime },b^{\prime }\right) \text{,}
\end{equation*}%
where%
\begin{eqnarray*}
a^{\prime } &=&\frac{\Pi _{1\leq s\leq k}\left(
i_{1}+...+i_{s}+j_{1}+...+j_{s}+ab\right) }{\Pi _{1\leq s\leq k}\left(
i_{1}+...+i_{s}+j_{1}+...+j_{s-1}+ab\right) }a, \\
b^{\prime } &=&\frac{\Pi _{1\leq s\leq k}\left(
i_{1}+...+i_{s}+j_{1}+...+j_{s-1}+ab\right) }{ab\Pi _{1\leq s\leq k-1}\left(
i_{1}+...+i_{s}+j_{1}+...+j_{s}+ab\right) }b\text{.}
\end{eqnarray*}%
The relations $\left[ \alpha ^{i},\beta ^{j}\right] =\left[ \alpha
^{j},\beta ^{i}\right] $ hold on the set $T$, for all integers $i,j$.
\end{lemma}

\begin{proof}
We calculate only the action of $\alpha ^{i_{1}}\beta ^{j_{1}}$:%
\begin{eqnarray*}
\left( a,b\right) &\rightarrow &^{\alpha ^{i_{1}}}\text{ \ }\left( a,\frac{%
i_{1}}{a}+b\right) =\left( a,\frac{i_{1}+ab}{ab}b\right) \\
&\rightarrow &^{\beta ^{j_{1}}}\text{ \ }\left( a+\frac{j_{1}}{i_{1}+ab}a,%
\frac{i_{1}+ab}{ab}b\right) \\
&=&\left( \frac{i_{1}+j_{1}+ab}{i_{1}+ab}a,\frac{i_{1}+ab}{ab}b\right) \text{%
.}
\end{eqnarray*}

It is direct to check that the first general non-trivial relation happens
for $k=3$:%
\begin{equation*}
j_{1}=-i_{1}-i_{2},j_{2}=i_{1},i_{3}=-i_{1}-i_{2},j_{3}=i_{2}\text{;}
\end{equation*}%
that is,%
\begin{equation*}
\alpha ^{i_{1}}\beta ^{-i_{1}-i_{2}}\alpha ^{i_{2}}\beta ^{i_{1}}\alpha
^{-i_{1}-i_{2}}\beta ^{i_{2}}=e
\end{equation*}

which in turn is equivalent to%
\begin{equation*}
\left[ \alpha ^{i},\beta ^{j}\right] =\left[ \alpha ^{j},\beta ^{i}\right] 
\text{.}
\end{equation*}
\end{proof}

\begin{lemma}
Let $charac\left( F\right) =2$, $a,b\in F$ such that $a\not=0\not=b,$ $%
ab\not=1$. Define $c=c\left( a,b\right) =\frac{ab}{1+ab}$. Then, $\alpha
^{2}=\beta ^{2}=e$ and for all integers $k,$%
\begin{eqnarray*}
\left( \alpha \beta \right) ^{k} &:&\left( a,b\right) \rightarrow \left(
c^{k}a,c^{-k}b\right) , \\
\left( \alpha \beta \right) ^{k}\beta &:&\left( a,b\right) \rightarrow
\left( c^{k-1}a,c^{-k}b\right) \text{.}
\end{eqnarray*}
\end{lemma}

The orbit of $\left( a,b\right) $ under the action of $\left\langle \alpha
,\beta \right\rangle $ has length $2.o(c)$. As both $ac^{i}\in F$ and $\dot{b%
}\dot{c}^{-i}\in \dot{F}$ invert $w=\left[ a,\dot{b}\right] $ for all $i\geq
0$, we conclude that the subgroup $\left\langle \left( 1+c^{i}\right)
a,\left( 1+\dot{c}^{i}\right) \dot{b}\mid i\not=0\right\rangle $ centralizes 
$w$. If $c$ satisfies a monic polynomial over $GF(2)$, which is the sum of
an odd number of monomials, then $a$ and $\dot{b}$ \ centralize $w$ and so $%
\left( a\dot{b}\right) ^{4}=e$.

\textbf{Proof of Theorem 10}. Let $a\not=0\not=b,$ $ab\not=1$. Since $c=%
\frac{ab}{1+ab}$ is a generator of the multiplicative group $F^{\#}$ we have 
$\left\{ \left( 1+c^{i}\right) a\mid i\geq 0\right\} =F^{\#}$ and so, $%
w=\left( a\dot{b}\right) ^{2}=\left[ a,\dot{b}\right] $ is central in $G$.
Therefore, $G$ is nilpotent of class at most $2$.\newline
\newline
\textbf{Example 5. }We obtain by using GAP:\newline
for $F=GF\left( 2^{3}\right) ,$ the group $G$ has order $2^{8}$ and
nilpotency class $2$;\newline
for $F=GF\left( 2^{4}\right) $ the group $G$ has order $2^{11}$ and
nilpotency class $2$;\newline
for $F=GF\left( 3^{3}\right) $ the group $G$ is abelian, isomorphic to $%
F\times F$.

\subsection{Group extension of $\protect\chi \left( A\right) $}

Let $\widetilde{A},\widetilde{B}$ be groups isomorphic to $A_{2,k}$, $k\geq
3 $, generated by $\left\{ a_{i}\mid 1\leq i\leq k\right\} $, $\left\{
b_{i}\mid 1\leq i\leq k\right\} $, respectively. Define $A=\left\langle
a_{i}\mid 2\leq i\leq k\right\rangle ,B=\left\langle b_{i}\mid 2\leq i\leq
k\right\rangle $ and let $f:A\rightarrow $ $B$ be the isomorphism extended
from the map $a_{i}\rightarrow b_{i}$ $\left( 2\leq i\leq k\right) $. Then, $%
G\left( A,B;f\right) $ is isomorphic to $\chi \left( A_{2,k-1}\right) $.
Both $\widetilde{A},\widetilde{B}$ are central extensions of $A,B$,
respectively. Define $f^{\ast }:\widetilde{A}\rightarrow \widetilde{B}$ by 
\begin{equation*}
a_{1}\rightarrow b_{1},h\rightarrow b_{1}h^{f},a_{1}h\rightarrow h^{f}\text{
for }h\in A^{\#}\text{;}
\end{equation*}%
this corresponds to choosing the bijections $\alpha :e\rightarrow
e,a_{1}\rightarrow b_{1}$, $\gamma :e\rightarrow b_{1},a_{1}\rightarrow e$.

We sketch below a proof that the group $G=G\left( \widetilde{A},\widetilde{B}%
;f^{\ast }\right) $ is metabelian, has order $2^{2^{k}+k-1}$ and has
nilpotency class $k$. Yet, $G$ is not isomorphic to $\chi \left(
A_{2,k}\right) $: for whereas the commutator subgroup of $\chi \left(
A_{2,k}\right) $ is of exponent $2$, $G^{\prime }$ is of exponent $4$.
Indeed, $G^{\prime }$ is generated by $\left[ a_{j}^{f},a_{1}\right] $ for $%
j>1$ (their number is $k-1$, each of order $2$) and by $\left[
a_{j_{1}}^{f},a_{j_{2}},a_{j_{3}},...,a_{j_{r}}\right] $ where $%
j_{1}>j_{2}>...>j_{r}>1$ (their number is $2^{k-1}-k$, each of order $4$).

We develop the proof in steps.

Let $t:\widetilde{A}\rightarrow \widetilde{B}$ be the natural isomorphism
extended from $a_{i}\rightarrow b_{i}$.

(1). The set $\left\{ \left( x,y^{t}\right) \mid x,y\in \widetilde{H}%
\right\} $ partitions under the substitutions $\alpha ,\beta $ into three
types of orbits. Let $u\not=w\in A$. Then orbits types are as follows:

(i) 
\begin{eqnarray*}
&&\{\left( a_{1},u^{t}\right) ,\left( u,u^{t}\right) ,\left( u,b_{1}\right)
,\left( a_{1}u,b_{1}\right) , \\
&&\left( a_{1}u,b_{1}u^{t}\right) ,\left( a_{1},b_{1}u^{t}\right) \}\text{;}
\end{eqnarray*}

(ii) 
\begin{eqnarray*}
&&\{\left( u,w^{t}\right) ,\left( a_{1}uw,w^{t}\right) ,\left(
a_{1}uw,u^{t}\right) ,\left( w,u^{t}\right) , \\
&&\left( w,b_{1}u^{t}w^{t}\right) ,\left( u,b_{1}u^{t}w^{t}\right) \}\text{;}
\end{eqnarray*}

(iii) 
\begin{eqnarray*}
&&\{\left( a_{1}u,b_{1}w^{t}\right) ,\left( a_{1}uw,b_{1}w^{t}\right)
,\left( a_{1}uw,b_{1}u^{t}\right) ,\left( a_{1}w,b_{1}u^{t}\right) , \\
&&\left( a_{1}w,b_{1}u^{t}w^{t}\right) ,\left( a_{1}u,b_{1}u^{t}w^{t}\right)
\}\text{.}
\end{eqnarray*}

(2). The following equalities hold for all $x,y,z\in \widetilde{A}$:

\begin{eqnarray*}
\left[ x,y^{t}\right] &=&\left[ y,x^{t}\right] , \\
\left[ x^{t},y,z^{t}\right] &=&\left[ z,y^{t},x\right] ^{-1}, \\
\left[ z,y^{t},x\right] &=&\left[ y,x^{t},z\right] \text{.}
\end{eqnarray*}

\textbf{Proof of (2)}. In each orbit we find $\left( x,y^{t}\right) ,\left(
y,x^{t}\right) $.

We conclude that for all $x,y,z\in \widetilde{A}$, 
\begin{eqnarray*}
\left[ x,\left( yz\right) ^{t}\right] &=&\left[ x,\left( zy\right) ^{t}%
\right] \\
&=&\left[ y,x^{t}\right] \left[ z,x^{t}\right] \left[ x,z^{t},y^{t}\right] ,
\\
\left[ x,\left( yz\right) ^{t}\right] &=&\left[ yz,x^{t}\right] \\
&=&\left[ y,x^{t}\right] \left[ y,x^{t},z\right] \left[ z,x^{t}\right] ,
\end{eqnarray*}

\begin{eqnarray*}
\left[ z,x^{t}\right] \left[ x,z^{t},y^{t}\right] &=&\left[ y,x^{t},z\right] %
\left[ z,x^{t}\right] , \\
\left[ x,z^{t},y^{t}\right] ^{\left[ x^{t},z\right] } &=&\left[ y,x^{t},z%
\right] , \\
\left[ x,z^{t},y^{t}\right] ^{\left[ z^{t},x\right] } &=&\left[ y,x^{t},z%
\right] , \\
\left[ y^{t},\left[ z^{t},x\right] \right] &=&\left[ y,x^{t},z\right] , \\
\left[ z^{t},x,y^{t}\right] &=&\left[ y,x^{t},z\right] ^{-1}
\end{eqnarray*}%
and

\begin{eqnarray*}
\left[ z^{t},x,y^{t}\right] &=&\left[ y,x^{t},z\right] ^{-1}=\left[ x,y^{t},z%
\right] ^{-1} \\
&=&\left[ x^{t},y,z^{t}\right] =\left[ z,y^{t},x\right] ^{-1}\text{.}
\end{eqnarray*}

(3). We will show that $a_{1}$ inverts $\left[ u,w^{t}\right] $ for all $%
u,w\in \widetilde{A}$.

\textbf{Proof of (3)}. Clearly $a_{1}$ inverts $\left[ a_{1},w^{t}\right] $.

In the calculations below, given a word $\ast \ast x\ast \ast $ we introduce
dots around $x$ as $\ast \ast .x.\ast \ast $ indicating that $x$ will be
substituted by $x^{f^{\ast }}xx^{f^{\ast }}$, if $x\in \widetilde{A}$, or by 
$x^{\left( f^{\ast }\right) ^{-1}}xx^{\left( f^{\ast }\right) ^{-1}}$, if $%
x\in \widetilde{B}$.

Let $u,w\in A$. Then, 
\begin{eqnarray*}
\left[ u,w^{t}\right] ^{a_{1}} &=&\left( a_{1}u\right) .w^{t}.uw^{t}a_{1} \\
&=&\left( uw\right) w^{t}.\left( a_{1}uw\right) .w^{t}a_{1} \\
&=&\left( uw\right) u^{t}\left( a_{1}uw\right) .u^{t}.a_{1} \\
&=&\left( uw\right) u^{t}wu^{t}u \\
&=&\left( wu^{t}wu^{t}\right) ^{u}=\left[ w,u^{t}\right] ^{u} \\
&=&\left[ u,w^{t}\right] ^{u}=\left[ u,w^{t}\right] ^{-1}\text{.}
\end{eqnarray*}%
Furthermore, 
\begin{eqnarray*}
\left[ a_{1}u,b_{1}w^{t}\right] ^{a_{1}}
&=&u.b_{1}w^{t}.a_{1}ub_{1}w^{t}a_{1} \\
&=&uw\left( b_{1}w^{t}\right) .a_{1}uw.b_{1}w^{t}a_{1} \\
&=&uw\left( b_{1}u^{t}\right) \left( a_{1}uw\right) .b_{1}u^{t}.a_{1} \\
&=&uw\left( b_{1}u^{t}\right) .a_{1}w.\left( b_{1}u^{t}\right) a_{1}u \\
&=&uw.b_{1}u^{t}w^{t}.\left( a_{1}w\right) \left( b_{1}u^{t}w^{t}\right)
a_{1}u \\
&=&\left( b_{1}u^{t}w^{t}\right) .a_{1}u.\left( b_{1}u^{t}w^{t}\right) a_{1}u
\\
&=&\left( b_{1}w^{t}\right) .a_{1}u.\left( b_{1}w^{t}\right) a_{1}u\text{.}
\\
&=&\left[ b_{1}w^{t},a_{1}u\right] \text{.}
\end{eqnarray*}

(4). We claim 
\begin{equation*}
\left[ \left[ z^{t},x\right] ,\left[ z^{t},x\right] ^{y}\right] =e,\left[
z^{t},x\right] ^{y}=\left[ z^{t},x\right] ^{y^{t}}
\end{equation*}

and the group $G$ is metabelian.

\textbf{Proof of (4)}. Apply $a_{1}$ to

\begin{equation*}
\left[ x,\left( yz\right) ^{t}\right] =\left[ x,\left( zy\right) ^{t}\right]
=\left[ x,y^{t}\right] \left[ x,z^{t}\right] \left[ x,z^{t},y^{t}\right] 
\text{.}
\end{equation*}

Then

\begin{eqnarray*}
\left[ \left( yz\right) ^{t},x\right] &=&\left[ y^{t},x\right] \left[ z^{t},x%
\right] \left[ x,z^{t},y^{t}\right] ^{a_{1}} \\
&=&\left[ y^{t}z^{t},x\right] =\left[ y^{t},x\right] ^{z^{t}}\left[ z^{t},x%
\right] \\
&=&\left[ y^{t},x\right] \left[ y^{t},x,z^{t}\right] \left[ z^{t},x\right] ,
\end{eqnarray*}

\begin{eqnarray*}
\left[ x,z^{t},y^{t}\right] ^{a_{1}} &=&\left[ y^{t},x,z^{t}\right] ^{\left[
z^{t},x\right] } \\
&=&\left[ z,x^{t},y\right] ^{-\left[ z^{t},x\right] } \\
&=&\left[ z^{t},x,y\right] , \\
\left[ x,z^{t},y^{t}\right] &=&\left[ z^{t},x,y\right] ^{a_{1}} \\
&=&\left[ x,z^{t},y\right]
\end{eqnarray*}

\begin{eqnarray*}
\left[ x,z^{t},y^{t}\right] &=&\left[ x,z^{t},y\right] , \\
\left[ x,z^{t},y\right] &=&\left[ x^{t},z,y^{t}\right] ^{-1} \\
&=&\left[ z^{t},x,y^{t}\right] ^{-1}\text{,} \\
\left[ \left[ x,z^{t}\right] ^{-1},y^{t}\right] &=&\left[ \left[ x,z^{t}%
\right] ,y^{t}\right] ^{-1}\text{.}
\end{eqnarray*}%
Thus,

\begin{equation*}
\left[ x,z^{t}\right] ^{y^{t}}=\left[ x,z^{t}\right] ^{y}
\end{equation*}%
and%
\begin{equation*}
\left[ x,z^{t}\right] \text{ commutes with }\left[ x,z^{t}\right] ^{y^{t}}%
\text{.}
\end{equation*}

Observe that%
\begin{equation*}
\left[ x,y^{t}\right] ^{uv^{t}}=\left[ x,y^{t}\right] ^{u^{t}v^{t}}=\left[
x,y^{t}\right] ^{\left( uv\right) ^{t}}=\left[ x,y^{t}\right] ^{uv}
\end{equation*}
and so, $\left[ x,y^{t}\right] ^{\left[ u,v^{t}\right] }=\left[ x,y^{t}%
\right] $. It follows that $G$ is metabelian.

(4.1). The commutator subgroup is generated by%
\begin{equation*}
\left[ a_{j}^{t},a_{i}\right] ^{w}\text{ where }j>i,w\in \left\langle
a_{r}\mid r\not=1,i,j\right\rangle \text{.}
\end{equation*}

We can improve upon the description of this generating set by using:%
\begin{eqnarray*}
\left[ a_{j}^{t},a_{i},a_{k}\right] ^{a_{i}}\left[ a_{k},a_{j}^{t},a_{i}%
\right] ^{a_{j}^{t}} &=&e, \\
\left[ a_{i},a_{j}^{t},a_{k}\right] \left[ a_{j},a_{k}^{t},a_{i}\right]
^{a_{j}} &=&e, \\
\left[ a_{i},a_{j}^{t},a_{k}\right] \left[ a_{k}^{t},a_{j},a_{i}\right] &=&e,
\end{eqnarray*}%
\begin{eqnarray*}
\left[ a_{j}^{t},a_{i},a_{k}\right] &=&\left[ a_{k}^{t},a_{j},a_{i}\right] =%
\left[ a_{j}^{t},a_{k},a_{i}\right] \\
&=&\left[ a_{i}^{t},a_{k},a_{j}\right] =\left[ a_{k}^{t},a_{i},a_{j}\right] =%
\left[ a_{i}^{t},a_{j},a_{k}\right] \text{.}
\end{eqnarray*}%
Therefore the generators have the form%
\begin{equation*}
\left[ a_{j_{1}}^{t},a_{j_{2}},a_{j_{3}},...,a_{j_{r}}\right] \text{ where }%
j_{1}\geq j_{2}\geq ...\geq j_{r}\text{.}
\end{equation*}%
If $j_{1}=j_{2}$ then 
\begin{equation*}
\left[ a_{j_{1}}^{t},a_{j_{2}}\right] =\left[ a_{j_{1}}^{t},a_{1}\right] ,
\end{equation*}%
\begin{equation*}
\left[ a_{j_{1}}^{t},a_{j_{2}},a_{j_{3}}\right] =\left[
a_{j_{1}}^{t},a_{1},a_{j_{3}}\right] =\left[ a_{j_{1}}^{t},a_{j_{3}},a_{1}%
\right] =\left[ a_{j_{1}}^{t},a_{j_{3}}\right] ^{2}\text{.}
\end{equation*}%
Thus $G^{\prime }$ is generated by:%
\begin{equation*}
\left[ a_{j}^{t},a_{1}\right] \text{ }\left( 1<j\leq k\right)
\end{equation*}%
(in total of $k-1$, each of order dividing $2$) and by 
\begin{equation*}
\left[ a_{j_{1}}^{t},a_{j_{2}},a_{j_{3}},...,a_{j_{r}}\right] \text{ for all 
}j_{1}>j_{2}>...>j_{r}>1
\end{equation*}%
(in total of $2^{k-1}-k$, each of order dividing $4$). Hence, $\left\vert
G^{\prime }\right\vert $ divides $%
2^{k-1}.4^{2^{k-1}-k}=2^{k-1+2^{k}-2k}=2^{2^{k}-k-1}$ and $\left\vert
G\right\vert $ divides $2^{2^{k}-k-1}2^{2k}=2^{2^{k}-1}2^{k}$. The
nilpotency class of $G$ is at most $k$ and the commutator of highest weight
is apparently\newline
$\left[ a_{k}^{t},a_{k-1},a_{k-2},...,a_{1}\right] $.

Rather than effecting the final construction, we just remark that
computations in GAP confirm the structural information obtained above for $%
k\leq 5$.

\subsection{A transposition}

Let $A=A_{2,k}$ and let $f$ correspond to a transposition. Since $SL\left(
k,2\right) $ is $2$-transitive on $A_{2,k}^{\#}$, any transposition of $%
A_{2,k}^{\#}$ is equivalent to $f$. Detailed analysis of $G\left(
A_{2,k};f\right) $ indicates that it has the same order as $\chi \left(
A\right) $, but not isomorphic to the latter. Again, this is confirmed by
computations in GAP.

\end{document}